\documentclass[10pt,twoside]{amsart}
\usepackage{fullpage}
\usepackage[english]{babel}
\usepackage{graphicx}
\usepackage[T1]{fontenc}
\usepackage{color}
\usepackage{bigints}
\usepackage{soul}
\usepackage{chemmacros}

\newtheorem{Ass}{Assumption}
\newtheorem{Thm}{Theorem}
\newtheorem{Prop}[Thm]{Proposition}

\newtheorem{Lem}{Lemma}

\newtheorem*{Prop*}{Proposition}
\newtheorem*{Cor*}{Corollary}
\newtheorem*{Thm*}{Theorem}
\usepackage{hyperref}
\usepackage[utf8]{inputenc}
\usepackage[english]{babel}

\usepackage{etoolbox}
\apptocmd{\sloppy}{\hbadness 10000\relax}{}{}
\hypersetup{
	colorlinks=true,
	linkcolor=blue,
	urlcolor=cyan,
	citecolor=red
}

\begin{document}
\title[Degenerate three species reaction-diffusion system]{Convergence to equilibrium for\\ a degenerate three species reaction-diffusion system}

\bibliographystyle{plain}
	
\author[Saumyajit Das]{Saumyajit Das}
\address{S.D.: Department of Mathematics, Indian Institute of Technology Bombay, Powai, Mumbai 400076 India.}
	\email{194099001@iitb.ac.in}

\author[Harsha Hutridurga]{Harsha Hutridurga}
\address{H.H.: Department of Mathematics, Indian Institute of Technology Bombay, Powai, Mumbai 400076 India.}
	\email{hutridurga@iitb.ac.in}

\maketitle

\begin{abstract}
In this work, we study a $3\times 3$ triangular reaction-diffusion system. Our main objective is to understand the long time behaviour of solutions to this reaction-diffusion system when there are degeneracies. More precisely, we treat cases when one of the diffusion coefficients vanishes while the other two diffusion coefficients stay positive. We prove convergence to equilibrium type results. In all our results, the constants appearing in the decay estimates are explicit.
\end{abstract}

\vspace{0.5cm}


	
\section{Introduction}\label{sec:introduction}
Reaction-diffusion equations are among the most widely used differential equations in applications. These equations govern the evolution (in time) of species concentrations at various spatial locations that are simultaneously diffusing and undergoing chemical reactions. We consider a reaction-diffusion system that concerns the three species $\mathcal{X}_1, \mathcal{X}_2$ and $\mathcal{X}_3$ involved in the following reversible reaction:
\begin{center}
\ch{$\mathcal{X}_1$ + $\mathcal{X}_2$ <>[ ][ ] $\mathcal{X}_3$}
\end{center}
The spatial domain is taken to be a bounded domain $\Omega\subset\mathbb{R}^N$ with $\mathrm C^{2+\nu}$ boundary with $\nu >0$. For the unknowns $a,b,c:[0,T)\times\Omega\to\mathbb{R}$, representing the concentrations of the three species $\mathcal{X}_1, \mathcal{X}_2, \mathcal{X}_3$ respectively, we consider the following system of differential equations:
\begin{equation}\label{eq:model}
\left\{
\begin{aligned}
\partial_{t}a - d_{a} \Delta a & = c-ab \qquad \qquad \mbox{ in } (0,T)\times\Omega,
\\
\partial_{t}b - d_{b} \Delta b & = c-ab\qquad \qquad \mbox{ in } (0,T)\times\Omega,
\\
\partial_{t}c - d_{c} \Delta c & =  ab-c\qquad \qquad \mbox{ in } (0,T)\times\Omega,
\\
\nabla_{x} a\cdot n=\nabla_{x} b\cdot n & =\nabla_{x} c\cdot n=0 \quad \mbox{ on } (0,T) \times \partial\Omega,
\\
a(0,\cdot)=a_0;\, b(0,\cdot)=b_0; \, & c(0,\cdot)=c_0 \qquad \qquad \mbox{ in } \Omega.
\end{aligned}
\right.
\end{equation}
Here $n(x)$ denotes the outward unit normal to $\Omega$ at the point $x\in\partial\Omega$. The initial data $a_{0}, b_{0}, c_{0}$ are taken to be smooth up to the closure of the domain and strictly positive. The diffusion coefficients $d_{a},d_{b},d_{c}$ are taken to be nonnegative. When all the diffusion coefficients are strictly positive (referred to, now onwards, as the non-degenerate setting), it is well-known that a unique strictly positive global-in-time $\mathrm{C}^{\infty}$ solution exists for the above model (see \cite{Pierre2003, Pie2010}), where the positiveness comes from the quasi-positive nature of the rate function \cite{rothe06, Amann1985, quittner2019superlinear}. The long time behaviour of solutions to \eqref{eq:model} in the non-degenerate setting was addressed in \cite{DF06} using the method of entropy. In this method, a Lyapunov functional (termed entropy) is found for the evolution equation. The negative of the time derivative of this entropy functional is referred to as the entropy dissipation functional. The entropy dissipation functional is then related back to the relative entropy via a functional inequality. This will then be followed by a Gr\"onwall type argument to deduce convergence of relative entropy to zero. A Czisz\'ar-Kullback-Pinsker type inequality that relates relative entropy and the $\mathrm L^1$-norm helps the authors in \cite{DF06} to prove the convergence to equilibrium in the $\mathrm L^1(\Omega)$-norm and to deduce that the decay is exponentially fast in time. In the proof of \cite{DF06}, uniform boundedness of the solution $(a,b,c)$ to \eqref{eq:model} plays an important role. This uniform bound is available in the non-degenerate case.

In this paper, we discuss the long time behaviour of solutions to a couple of degenerate cases of the three species model \eqref{eq:model}. Our work is heavily inspired by \cite{EMT20} which dealt with a four species degenerate model where one of the species stops diffusing. The authors in \cite{EMT20} demonstrated a so-called \emph{indirect diffusion effect} wherein an effective diffusion is felt by the non-diffusive species, thanks to the interplay between the diffusion from diffusive species and the underlying reversible reaction. Our work demonstrates that a similar indirect diffusion effect is present in the above three species model in the presence of degeneracies. The two degenerate reaction-diffusion systems that we study in this article are:
\begin{equation}\label{eq:model 1}
\left\{
\begin{aligned}
\partial_{t}a - d_{a} \Delta a & = c-ab \qquad \qquad \mbox{ in } (0,T)\times\Omega,
\\
\partial_{t}b & = c-ab\qquad \qquad \mbox{ in } (0,T)\times\Omega,
\\
\partial_{t}c- d_{c} \Delta c & =  ab-c\qquad \qquad \mbox{ in } (0,T)\times\Omega,
\\
\nabla_{x} a\cdot n=\nabla_{x} b \cdot n & =\nabla_{x} c\cdot n=0 \quad \mbox{ on } (0,T) \times \partial\Omega,
\\
a(0,\cdot)=a_0;\, b(0,\cdot)=b_0; \, & c(0,\cdot)=c_0 \qquad \qquad \mbox{ in } \Omega,
\end{aligned}
\right.
\end{equation}
and
\begin{equation}\label{eq:model 11}
\left\{
\begin{aligned}
\partial_{t}a- d_{a} \Delta a & = c-ab \qquad \qquad \mbox{ in } (0,T)\times\Omega,
\\
\partial_{t}b-d_b \Delta b & = c-ab\qquad \qquad \mbox{ in } (0,T)\times\Omega,
\\
\partial_{t}c & =  ab-c\qquad \qquad \mbox{ in } (0,T)\times\Omega,
\\
\nabla_{x} a\cdot n=\nabla_{x} b\cdot n & =\nabla_{x} c\cdot n=0 \quad \mbox{ on } (0,T) \times \partial\Omega,
\\
a(0,\cdot)=a_0;\, b(0,\cdot)=b_0; \, & c(0,\cdot)=c_0 \qquad \qquad \mbox{ in } \Omega.
\end{aligned}
\right.
\end{equation}
Here, the initial data are assumed to be smooth and strictly positive. The diffusion coefficients $d_a,d_c$ in \eqref{eq:model 1} and the diffusion coefficients $d_a,d_b$ in \eqref{eq:model 11} are assumed to be strictly positive. The existence of a smooth global-in-time positive solution to \eqref{eq:model 1} was proved in \cite[Theorem 3.1]{DF15}. The existence of a smooth global-in-time positive solution to \eqref{eq:model 11} was proved in \cite[Theorem 3.2]{DF15} if the dimension $N\le3$. The authors in \cite{DF15}, however, prove the existence of a weak global-in-time positive solution to \eqref{eq:model 11} in any arbitrary dimension. 

The process of establishing various estimates in \cite{EMT20} has greatly motivated us to establish similar kind of estimates in dimension larger than three for the degenerate case \eqref{eq:model 1} corresponding to the vanishing of the diffusion coefficient $d_b$. We further use various Neumann Green's function results from \cite{rothe06, Morra83, ML15, FMT20} to get various estimates specifically for dimensions less than four.

Observe that both the models \eqref{eq:model 1} and \eqref{eq:model 11} satisfy the following mass conservation properties:
\begin{align} 
0<M_{1} \vert \Omega \vert =: \int_{\Omega} \big( a(t,x)+c(t,x) \big)\, {\rm d}x = \int_{\Omega} \big( a(0,x)+c(0,x) \big)\, {\rm d}x,\label{mass conserve 1}
\\
0<M_{2} \vert \Omega \vert =: \int_{\Omega} \big( b(t,x)+c(t,x) \big)\, {\rm d}x = \int_{\Omega} \big( b(0,x)+c(0,x) \big)\, {\rm d}x.\label{mass conserve 2} 
\end{align}
An homogenous (constant) equilibrium state $(a_\infty,b_\infty, c_\infty)$ associated with these models should also satisfy the above conservation properties. Hence we should have
\begin{align}\label{eq:equi-state-1}
a_{\infty}+c_{\infty}=M_{1}, \qquad b_{\infty}+c_{\infty}=M_{2}.
\end{align}
Moreover, at equilibrium, the rate function should vanish, i.e.
\begin{align}\label{eq:equil-relation}
c_{\infty}=a_{\infty}b_{\infty}.
\end{align}
The relations \eqref{eq:equi-state-1} and \eqref{eq:equil-relation} put together leads to a quadratic equation for $c_\infty$ whose only admissible non-negative solution is
\begin{align}\label{eq:equi-state-2}
c_{\infty} =\frac{1}{2}(1+M_{1}+M_{2})-\frac{1}{2}\sqrt{(1+M_{1}+M_{2})^2-4M_{1}M_{2}}.
\end{align}
The corresponding $a_\infty$ and $b_\infty$ can be computed using \eqref{eq:equi-state-1}. Now onwards, we will be considering this unique homogeneous equilibrium state $(a_\infty,b_\infty, c_\infty)$. Next, we list a bunch of notations which will be used throughout this manuscript. These notations are inspired by those used in \cite{DF06}.
\begin{itemize}
\item The square roots of the species concentrations are denoted as
\[
A:= \sqrt{a}, \quad B:= \sqrt{b}=B, \quad C:=\sqrt{c}.
\]
\item The square roots of the homogeneous equilibrium states are denoted as
\[
A_{\infty} := \sqrt{a_{\infty}}, \quad B_{\infty} := \sqrt{b_{\infty}}, \quad C_{\infty} := \sqrt{c_{\infty}}.
\]
\item The average of a function $f:\Omega\to \mathbb{R}$ is denoted as
\[
\overline{f} := \frac{1}{\vert \Omega\vert}\int_{\Omega}f(x)\, {\rm d}x.
\]
\item The deviations of the square roots of species concentrations from their averages are denoted as
\[
\delta_{A} := A - \overline{A}, \qquad \delta_{B} := B - \overline{B}, \qquad \delta_{C} := C - \overline{C}.
\]
\item The parabolic cylinders are denotes as
\[
\Omega_{\tau,T} := (\tau,T)\times\Omega \qquad \mbox{ for }\, 0\leq\tau<T.
\]
\item The lateral boundary of the parabolic cylinders are denoted as
\[
\partial\Omega_{\tau,T} := (\tau,T)\times\partial\Omega \qquad \mbox{ for }\, 0\leq\tau<T.
\]
\end{itemize} 
We consider the following entropy functional associated with \eqref{eq:model 1} and \eqref{eq:model 11}:
\begin{align}\label{entropy} 
E(a,b,c) := \int_{\Omega} \Big((a(\ln a-1)+1)+(b(\ln b-1)+1)+(c(\ln c-1)+1)\Big) {\rm d}x.
\end{align}	
While studying \eqref{eq:model 1} in dimensions $N\ge4$, we have been able to arrive at large time asymptotics of the solution only under certain closeness assumption on the non-zero diffusion coefficients. The precise assumption is the following:
\begin{Ass}
The non-zero diffusion coefficients $d_a$ and $d_c$ are said to satisfy the closeness assumption if
\begin{align}\label{closeness condition imp}
\vert{d_a-d_c}\vert < \frac{2}{C^{PRC}_{{\frac{d_a+d_c}{2}},p'}}
\qquad 
\text{ and } 
\qquad \frac{\vert d_a-d_c \vert}{d_a+d_c}<\frac{1}{C_{SOR}(\Omega,N,p')},
\end{align}
where the constants $C^{PRC}_{{\frac{d_a+d_c}{2}},p'}$ and $C_{SOR}(\Omega,N,p')$, are the parabolic regularity constant (see Theorem \ref{PE} in the Appendix) and the second order regularity constant (see Theorem \ref{estimation 1} in the Appendix), respectively.
	\end{Ass}
Now we are going to state our two main results.
\begin{Thm}\label{convergence theorem 1}
For $N\geq4$, let $(a,b,c)$ be the solution to the degenerate system \eqref{eq:model 1}. Let $(a_\infty,b_\infty,c_\infty)$ be the associated equilibrium state given by \eqref{eq:equi-state-1}-\eqref{eq:equi-state-2}. Let the nonzero diffusion coefficients $d_a,d_c$ satisfy the closeness condition \eqref{closeness condition imp}. Then, for any given positive $\varepsilon \ll1$, there exists a time $T_{\varepsilon}$ and two positive constants $\mathcal{S}_1$ and $\mathcal{S}_2$ such that for $t\ge T_\varepsilon$, we have
\[
\frac{1}{2M_{1}}\Vert a-a_{\infty}\Vert_{L^1(\Omega)}^{2}+\frac{1}{2M_{2}}\Vert b-b_{\infty}\Vert_{L^1(\Omega)}^{2}+ \frac{1}{(M_{1}+M_{2})}\Vert c-c_{\infty}\Vert_{L^1(\Omega)}^{2}
\leq \frac{(9+2\sqrt{2})}{(3+2\sqrt{2})\vert \Omega \vert} \mathcal{S}_1 e^{-\mathcal{S}_2(1+t)^{\frac{1-\epsilon}{N-1}}}.
\]
For $N<4$, let $(a,b,c)$ be the solution to the degenerate system \eqref{eq:model 1}. Let $(a_\infty,b_\infty,c_\infty)$ be the associated equilibrium state given by \eqref{eq:equi-state-1}-\eqref{eq:equi-state-2}. Then, for any given positive $\varepsilon \ll1$, there exists a time $T_{\varepsilon}$ and two positive constants $\mathcal{S}_3$ and $\mathcal{S}_4$ such that for $t\ge T_\varepsilon$, we have
\[
\frac{1}{2M_{1}}\Vert a-a_{\infty}\Vert_{L^1(\Omega)}^{2}+\frac{1}{2M_{2}}\Vert b-b_{\infty}\Vert_{L^1(\Omega)}^{2}+ \frac{1}{(M_{1}+M_{2})}\Vert c-c_{\infty}\Vert_{L^1(\Omega)}^{2}
\leq \frac{(9+2\sqrt{2})}{(3+2\sqrt{2})\vert \Omega \vert} \mathcal{S}_3 \, e^{-\mathcal{S}_4(1+t)^{\frac{1-\epsilon}{6}}}.
\]
\end{Thm}
Observe that our above result is unconditional when the dimension $N\le 3$, i.e. there is no closeness assumption on the non-zero diffusion coefficients $d_a$ and $d_c$. Furthermore, the constants $M_1$ and $M_2$ are determined by the initial data (see \eqref{mass conserve 1} and \eqref{mass conserve 2}).
	\begin{Thm}\label{convergence theorem 2}
Let $N\leq3$ and let $(a,b,c)$ be the solution to the degenerate system \eqref{eq:model 11}. Let $(a_\infty,b_\infty,c_\infty)$ be the associated equilibrium state given by \eqref{eq:equi-state-1}-\eqref{eq:equi-state-2}. Then, for any given positive $\varepsilon \ll1$, there exists a time $T_{\varepsilon}$ and two positive constants $\mathcal{S}_5$ and $\mathcal{S}_6$ such that for $t\ge T_\varepsilon$, we have
\[
\frac{1}{2M_{1}}\Vert a-a_{\infty}\Vert_{L^1(\Omega)}^{2}+\frac{1}{2M_{2}}\Vert b-b_{\infty}\Vert_{L^1(\Omega)}^{2}+ \frac{1}{(M_{1}+M_{2})}\Vert c-c_{\infty}\Vert_{L^1(\Omega)}^{2}
\leq \frac{(9+2\sqrt{2})}{(3+2\sqrt{2})\vert \Omega \vert} \mathcal{S}_5 \, e^{-\mathcal{S}_6(1+t)^{\frac{2-\epsilon}{3}}}.
\]
\end{Thm}
In the reminder of this introduction, we briefly describe our strategy to arrive at the aforementioned large time behaviour via the entropy method. We start by defining the entropy dissipation functionals associated with the degenerate models \eqref{eq:model 1} and \eqref{eq:model 11}:
\begin{equation}\label{entropy dis} 
\left \{
\begin{aligned}
\text{For} \ d_b & = 0,
\\
D(a,b,c) := & \, 4\, d_{a}\int_{\Omega}\left\vert\nabla \sqrt{a}\right\vert^{2}\, {\rm d}x + \, 4\, d_{c}\int_{\Omega}\left\vert\nabla \sqrt{c}\right\vert^2\, {\rm d}x  + \int_{\Omega}(ab-c)\ln\left(\frac{ab}{c}\right)\, {\rm d}x,
\\
\text{For} \ d_c & = 0,
\\
D(a,b,c) := &\, 4\,  d_{a}\int_{\Omega}\left\vert\nabla \sqrt{a}\right\vert^2\, {\rm d}x + \, 4\, d_{b}\int_{\Omega}\left\vert\nabla \sqrt{b}\right\vert^2\, {\rm d}x + \int_{\Omega}(ab-c)\ln\left(\frac{ab}{c}\right)\, {\rm d}x.
\end{aligned}
\right.
\end{equation}
Note that both the dissipation functionals are positive. We recall here an algebraic inequality:
\[ 
(x-y)(\ln{x}-\ln{y}) \geq 4(\sqrt{x}-\sqrt{y})^2, \ \ \forall x,y \geq 0.
\]
This algebraic inequality gives a lower bound on the last term of the above entropy dissipation functionals whereas the classical Poincar\'e inequality gives a lower bound for the gradient terms. More precisely, we have
\begin{equation}\label{new label d_a,d_c L2}
\left \{
\begin{aligned}
D(a,b,c) \geq &  \frac{ 4\, d_a}{P(\Omega)} \Vert  \delta_{A} \Vert^2_{\mathrm{L}^2(\Omega)}+ \frac{ 4\, d_c}{P(\Omega)} \Vert  \delta_{C} \Vert^2_{\mathrm{L}^2(\Omega)}+4\Vert AB-C \Vert^2_{\mathrm{L}^2(\Omega)}, \qquad \mbox{ for }\, d_b=0,\\
D(a,b,c) \geq &  \frac{ 4\, d_a}{P(\Omega)} \Vert  \delta_{A} \Vert^2_{\mathrm{L}^2(\Omega)}+ \frac{ 4\, d_b}{P(\Omega)} \Vert  \delta_{B} \Vert^2_{\mathrm{L}^2(\Omega)}+4\Vert AB-C \Vert^2_{\mathrm{L}^2(\Omega)},\qquad \mbox{ for }\, d_c=0,\\
\end{aligned}
\right .
\end{equation}
where $P(\Omega)$ is the Poincar\'e constant of the domain $\Omega$ (see Lemma \ref{Poincare-Wirtinger} in the Appendix). Observe that the deviation term $\Vert  \delta_{B} \Vert^2_{\mathrm{L}^2(\Omega)}$ is missing in the above lower bound corresponding to the degenerate case $d_b=0$ and the deviation term $\Vert  \delta_{C} \Vert^2_{\mathrm{L}^2(\Omega)}$ is missing in the lower bound corresponding to the degenerate case $d_c=0$. Note that, for the dissipation functional of the degenerate system \eqref{eq:model 1}, the Poincar\'e inequality actually yields
\begin{equation*}
\left \{
\begin{aligned}
D(a,b,c) \geq & \frac{4\, d_a}{P(\Omega)} \Vert \delta_{A} \Vert^2_{\mathrm{L}^{\frac{2N}{N-2}}(\Omega)}+\frac{4\, d_c}{P(\Omega)}\Vert \delta_{C} \Vert^2_{\mathrm{L}^{\frac{2N}{N-2}}(\Omega)}+4\Vert AB-C \Vert^2_{\mathrm{L}^2(\Omega)} \qquad \mbox{ for }\, N\geq 4,
\\
D(a,b,c)\geq&  \frac{4\, d_a}{P(\Omega)} \Vert \delta_{A} \Vert^2_{\mathrm{L}^6(\Omega)}+\frac{4\, d_c}{P(\Omega)}\Vert \delta_{C} \Vert^2_{\mathrm{L}^6(\Omega)}+4\Vert AB-C \Vert^2_{\mathrm{L}^2(\Omega)}  \qquad \qquad \quad \mbox{ for }N\le3.
\end{aligned} 
\right .
\end{equation*}
Similarly, for the dissipation functional of the degenerate system \eqref{eq:model 11}, we have
\begin{equation*}
\left \{
\begin{aligned}
D(a,b,c)\geq & \frac{4\, d_a}{P(\Omega)} \Vert \delta_{A} \Vert^2_{\mathrm{L}^\frac{2N}{N-2}(\Omega)}+\frac{4\, d_b}{P(\Omega)}\Vert \delta_{B} \Vert^2_{\mathrm{L}^\frac{2N}{N-2}(\Omega)}+4\Vert AB-C \Vert^2_{\mathrm{L}^2(\Omega)} \qquad \mbox{ for } N\geq 4,
\\
D(a,b,c)\geq & \frac{4\, d_a}{P(\Omega)} \Vert \delta_{A} \Vert^2_{\mathrm{L}^6(\Omega)}+\frac{4\, d_b}{P(\Omega)}\Vert \delta_{B} \Vert^2_{\mathrm{L}^6(\Omega)}+4\Vert AB-C \Vert^2_{\mathrm{L}^2(\Omega)} \qquad \qquad \quad \mbox{ for }N\le3.
\end{aligned}
\right .
\end{equation*}
Note that, if $(a,b,c)$ solves either of the system \eqref{eq:model 1} or \eqref{eq:model 11}, we have
\[ 
\frac{d}{dt}\Big(E(a,b,c)-E(a_{\infty},b_{\infty},c_{\infty})\Big)= -D(a,b,c),
\]
i.e. the relative entropy (relative with respect to the homogeneous equilibrium state) is non-increasing in time. The relative entropy has the following expression:
\begin{align*}
E(a,b,c)-E(a_{\infty},b_{\infty},c_{\infty}) = \int_{\Omega}\big(a \ln{a}-a-a_{\infty}& \ln{a_{\infty}}+a_{\infty}\big)\, {\rm d}x+\int_{\Omega}\big(b \ln{b}-b-b_{\infty} \ln{b_{\infty}}+b_{\infty}\big)\, {\rm d}x
\\
& +\int_{\Omega}\big(c \ln{c}-c-c_{\infty} \ln{c_{\infty}}+c_{\infty}\big)\, {\rm d}x.
\end{align*}
The above expression rewrites as
\begin{equation}\label{Entropy Gamma function relation 3}
\begin{aligned}
E(a,b,c)-E(a_{\infty},b_{\infty},c_{\infty}) =\int_{\Omega} \left(a \ln{\left(\frac{a}{a_{\infty}}\right)}-a+a_{\infty}\right)\, {\rm d}x + \int_{\Omega}& \left(b \ln{\left(\frac{b}{b_{\infty}}\right)}-b+b_{\infty}\right)\, {\rm d}x 
\\
&+\int_{\Omega} \left(c \ln{\left(\frac{c}{c_{\infty}}\right)}-c+c_{\infty}\right)\, {\rm d}x.
\end{aligned}
\end{equation}
Define a function $\Gamma:(0,\infty)\times(0,\infty)\to\mathbb{R}$ as follows:
\begin{equation}
\Gamma(x,y) :=
\left\{
\begin{aligned}
&\frac{x \ln\left(\frac{x}{y}\right)-x+y}{\left(\sqrt{x}-\sqrt{y}\right)^2}  \qquad \mbox{ for }x\not=y,
\\
&2 \qquad  \qquad \qquad \qquad \ \ \ \mbox{ for }x=y.
\end{aligned}\right.
\end{equation}
It can be shown (see \cite[Lemma 2.1, p.162]{DF06} for details) that the above defined function satisfies the following bound:
\begin{align*}
\Gamma(x,y)\leq C_{\Gamma}\max\left\{1,\ln\left(\frac{x}{y}\right)\right\}
\end{align*}
for some positive constant $C_{\Gamma}$. Note that using the function $\Gamma$ defined above, the relative entropy can be rewritten as 
\[
E(a,b,c)-E(a_{\infty},b_{\infty},c_{\infty})= \int_{\Omega} \Gamma(a,a_{\infty})(A-A_{\infty})^2+\int_{\Omega} \Gamma(b,b_{\infty})(B-B_{\infty})^2+\int_{\Omega} \Gamma(c,c_{\infty})(C-C_{\infty})^2.
\]
Using the aforementioned bound for $\Gamma$, we obtain
\begin{align*}
	E(a,b,c)-E(a_{\infty},b_{\infty},c_{\infty}) 
	\leq  C_{\Gamma}& \max\big\{1,\ln{(\Vert a\Vert_{\mathrm{L}^{\infty}(\Omega)}+1)}+\vert \ln{a_{\infty}}\vert\big\}\Vert A-A_{\infty}\Vert_{\mathrm{L}^2(\Omega)}^2 
	\\+& C_{\Gamma} \max\big\{1,\ln{(\Vert b\Vert_{\mathrm{L}^{\infty}(\Omega)}+1)}+\vert \ln{b_{\infty}}\vert\big\}\Vert B-B_{\infty}\Vert_{\mathrm{L}^2(\Omega)}^2
	\\+& C_{\Gamma} \max\big\{1,\ln{(\Vert c\Vert_{\mathrm{L}^{\infty}(\Omega)}+1)}+\vert \ln{c_{\infty}}\vert\big\}
	\Vert C-C_{\infty}\Vert_{\mathrm{L}^2(\Omega)}^2.
\end{align*}
We derive the growth rate of the solutions corresponding to \eqref{eq:model 1} and \eqref{eq:model 11} in Lemma \ref{L^infty} and in Proposition \ref{L^infty d_c} respectively. Both these results assert that the growth of the solutions can at most be polynomial in time. This leads to
\begin{align*}
E(a,b,c)-E(a_{\infty},b_{\infty},c_{\infty})
\leq   C_{1}(1+t)^{\varepsilon}\Big(\Vert A-A_{\infty}\Vert_{\mathrm{L}^2(\Omega)}^2+\Vert B-B_{\infty}\Vert_{\mathrm{L}^2(\Omega)}^2+\Vert C-C_{\infty}\Vert_{\mathrm{L}^2(\Omega)}^2\Big),
\end{align*}
for all $t\geq T_{\varepsilon}$, where $\varepsilon\ll 1$ is some positive quantity and $T_{\varepsilon}$ depends on $\varepsilon$. Furthermore, the positive constant $C_1$ depends only on the initial data, the domain $\Omega$ and the dimension $N$.

Observe that, other than the logarithm of the growth of the supremum norm of the solution, the growth of relative entropy depends on the $\mathrm{L}^2(\Omega)$ norm of the deviation of the $(A,B,C)$ from $(A_{\infty},B_{\infty},C_{\infty})$. Recall from \eqref{new label d_a,d_c L2} that the dissipation functional is also related to the $\mathrm{L}^2(\Omega)$ norm of the deviations $\delta_{A}$, $\delta_{B}$ and $\delta_{C}$. The following observation holds for all the three species. However, we choose to show it for the term $A$:
\[
\Vert A-A_{\infty}\Vert_{\mathrm{L}^2(\Omega)}^2 \leq 3\left( \Vert A-\overline{A}\Vert_{\mathrm{L}^2(\Omega)}^2+\Vert \overline{A}-\sqrt{\overline{A^2}}\Vert_{\mathrm{L}^2(\Omega)}^2+\Vert \sqrt{\overline{A^2}}-A_{\infty}\Vert_{\mathrm{L}^2(\Omega)}^2\right). 
\]
The following observation says that the first term on the right hand side dominates the second term:
\begin{align*}
	\Vert \overline{A}-\sqrt{\overline{A^2}}\Vert_{\mathrm{L}^2(\Omega)}^2= &\vert\Omega\vert \big\vert
	\overline{A}-\sqrt{\overline{A^2}} \big\vert^2 = \vert\Omega\vert \left( \overline{A^2}+\overline{A}^2-2\overline{A}\sqrt{\overline{A^2}} \right)\\ &
	\leq\vert\Omega\vert \left( \overline{A^2}-\overline{A}^2\right) \leq \Vert A-\overline{A}\Vert_{\mathrm{L}^2(\Omega)}^2.
\end{align*}
Here we used the fact that $\displaystyle{\overline{A}\leq \sqrt{\overline{A^2}}}$, thanks to H\"older inequality. Hence
\[
\Vert A-A_{\infty}\Vert_{\mathrm{L}^2(\Omega)}^2 \leq 6\left( \Vert A-\overline{A}\Vert_{\mathrm{L}^2(\Omega)}^2+\Vert \sqrt{\overline{A^2}}-A_{\infty}\Vert_{\mathrm{L}^2(\Omega)}^2\right). 
\]
This helps us deduce
\begin{align*}
\Vert A-A_{\infty}\Vert_{\mathrm{L}^2(\Omega)}^2+\Vert B-B_{\infty}\Vert_{\mathrm{L}^2(\Omega)}^2&+\Vert C-C_{\infty}\Vert_{\mathrm{L}^2(\Omega)}^2
\leq  6\left(\Vert A-\overline{A}\Vert_{\mathrm{L}^2(\Omega)}^2+\Vert B-\overline{B}\Vert_{\mathrm{L}^2(\Omega)}^2+\Vert C-\overline{C}\Vert_{\mathrm{L}^2(\Omega)}^2 \right)\\
+& 6\left(\Vert \sqrt{\overline{A^2}}-A_{\infty}\Vert_{\mathrm{L}^2(\Omega)}^2+\Vert \sqrt{\overline{B^2}}-B_{\infty}\Vert_{\mathrm{L}^2(\Omega)}^2+\Vert \sqrt{\overline{C^2}}-C_{\infty}\Vert_{\mathrm{L}^2(\Omega)}^2 \right).
\end{align*} 
Next we borrow a result from \cite{DF06}, which says that there exists $C_{EB}>0$, depending only on the domain and the equilibrium state $(a_{\infty},b_{\infty},c_{\infty})$ such that
\begin{equation}\label{ED}
\begin{aligned}
\Vert \sqrt{\overline{A^2}}-A_{\infty} \Vert_{L^2(\Omega)}^2+&\Vert \sqrt{\overline{B^2}}-B_{\infty} \Vert_{L^2(\Omega)}^2+\Vert \sqrt{\overline{C^2}}-C_{\infty} \Vert_{L^2(\Omega)}^2 
\\
\leq & C_{EB} \left (   \Vert  \delta_{A} \Vert_{L^{2}(\Omega)}^2+\Vert  \delta_{B} \Vert_{L^{2}(\Omega)}^2+\Vert  \delta_{C} \Vert_{L^{2}(\Omega)}^2+\Vert C-AB\Vert_{L^2(\Omega)}^2\right).
\end{aligned}
\end{equation}
This indicates that to relate entropy with  entropy dissipation, we need to relate dissipation with the missing term $\Vert \delta_{B} \Vert^2_{\mathrm{L}^2(\Omega)}$ for the degenerate case \eqref{eq:model 1}  and similarly with the missing term $\Vert \delta_{C} \Vert^2_{\mathrm{L}^2(\Omega)}$ for the degenerate case \eqref{eq:model 11}. 

For the degenerate system \eqref{eq:model 1}, the relation between the missing term and the entropy dissipation is shown in Proposition \ref{main relatio} for dimensions $N\geq 4$. For dimensions $N\leq 3$, we obtain the relation in Proposition \ref{2.0.2}. More precisely, we have obtained
\begin{equation*}
\left \{
\begin{aligned}
D(a,b,c) \geq & \, \hat C(1+t)^{-\frac{N-2}{N-1}}\Vert \delta_{B}\Vert_{\mathrm{L}^2(\Omega)}^2 \qquad  \mbox{ for } N\geq 4,
\\
D(a,b,c) \geq & \, \hat C(1+t)^{-\frac{5}{6}}\Vert \delta_{B}\Vert_{\mathrm{L}^2(\Omega)}^2 \qquad \quad\mbox{ for }  N\le 3.
\end{aligned}
\right .
\end{equation*}
where $\hat{C}$ is some positive constant, independent of time.

In order to obtain the above lower bound, we require an estimate on the growth (in time) of the $\mathrm{L}^{\frac{N}{2}}(\Omega)$ norm of the degenerate species $b$ for dimension $N\geq 4$ and for dimension $N\leq 3$, we require a similar growth estimate of the $\mathrm{L}^{\frac{3}{2}}(\Omega)$ norm of the degenerate species $b$. These estimates are obtained in Lemma \ref{N/2 estimate}, Proposition \ref{N=3 main} and Lemma \ref{2.0.3}. More precisely, we have obtained
\begin{equation*}
	\left \{
	\begin{aligned}
		\Vert b\Vert_{\mathrm{L}^{\frac{N}{2}}(\Omega)} & \leq \hat{K}(1+t)^{\frac{N-2}{N-1}} \qquad N \geq 4,\\
		\Vert b \Vert_{\mathrm{L}^{\frac{3}{2}}(\Omega)} & \leq \hat{K} (1+t)^{\frac{5}{6}} \qquad N=1,2,3,
	\end{aligned}
	\right .
\end{equation*}
where $\hat{K}$ is a positive constant independent of time. For dimension $N\leq 3$, the above integral estimation of the species $b$ follows from Gagliardo-Nirenberg inequality whereas  for dimension  $N\geq 4$,  we need closeness assumption \eqref{closeness condition imp} on the other two non-zero diffusion coefficients  $d_a$ and $d_c$. Closeness assumption further helps us to estimate $\mathrm{L}^p$ integral growth of  $a$ and $c$ on a parabolic cylinder with unit height, for some large exponent $p$. It turns out if  $d_a$ and  $d_c$  satisfy the closeness condition \eqref{closeness condition imp} for dimension  $N\geq 4$, then there exists a positive constant $C_0>0$, depending on $p$ and independent of time, such that
\[ 
\Vert a \Vert_{\mathrm{L}^{p}((\tau,\tau+1)\times\Omega)}+\Vert c \Vert_{\mathrm{L}^{p}((\tau,\tau+1)\times\Omega)}\leq C_0 \qquad \forall \tau\geq 0, \ p>N.
\]
The above result is proved in Lemma \ref{L^p with time}. For dimension $N\geq 4$, this will lead  us to prove our key $\mathrm{L}^{\frac{N}{2}}(\Omega)$ estimate. These strategies are inspired from the article \cite{EMT20}, where the authors study existence and large time behaviour of a particular $4\times4$ quadratic degenerate reaction-diffusion system.

For the degenerate system \eqref{eq:model 11},  we will relate the entropy dissipation functional with the missing term $\Vert \delta_{C} \Vert^2_{\mathrm{L}^2(\Omega)}$. In this article, for the degenerate case \eqref{eq:model 11}, we analyze the decay of entropy in dimension  $N\leq 3$. We will establish our result only for dimension $N=3$. For dimension $N=1,2$, all the calculations are similar. In Proposition \ref{missing theorem} We establish the following relation between the entropy dissipation and the missing term $\Vert \delta_{C} \Vert^2_{\mathrm{L}^2(\Omega)}$
\[
D(a,b,c) \geq \hat{C}(1+t)^{-\frac{1}{3}}\left(\Vert A-\overline{A}\Vert_{\mathrm{L}^2(\Omega)}^2+\Vert B-\overline{B}\Vert_{\mathrm{L}^2(\Omega)}^2+\Vert C-\overline{C}\Vert_{\mathrm{L}^2(\Omega)}^2\right).
\]
In order to arrive at the above estimate we need a particular integral estimate of species $a$, which we will establish in Lemma \ref{L^p estimate d_c}. The particular relation is the following
\[
\Vert a\Vert_{\mathrm{L}^\frac{3}{2}(\Omega)} \leq  \tilde{C}(1+t)^{\frac{1}{3}} \qquad \forall t \geq 0
\]
where $\tilde{C}>0$ a constant, independent of time. This is the key estimate for the degenerate case \eqref{eq:model 11}. To establish this we will use a particular integral estimate in Lemma \ref{Lem:L1}. 

These results, along with the Gr\"onwall inequality, will lead us to the conclusion that the entropy functional decays sub-exponentially fast in time for both the degenerate systems \eqref{eq:model 1} and \eqref{eq:model 11}. On the other hand, an application of the Czisz\'ar-Kullback-Pinsker inequality yields the following lower bound on the relative entropy (see \cite{DF06} for details):
\begin{flalign*}
	E(a,b,c) -E(a_{\infty},b_{\infty},c_{\infty})&\geq  \frac{(3+2\sqrt{2})\vert \Omega \vert}{2M_{1}(9+2\sqrt{2})}\Vert a-a_{\infty}\Vert_{\mathrm{L}^1(\Omega)}^{2} \\ +&\frac{(3+2\sqrt{2})\vert \Omega \vert}{2M_{2}(9+2\sqrt{2})}\Vert b-b_{\infty}\Vert_{\mathrm{L}^1(\Omega)}^{2}+ \frac{(3+2\sqrt{2})\vert \Omega \vert}{(M_{1}+M_{2})(9+2\sqrt{2})}\Vert c-c_{\infty}\Vert_{\mathrm{L}^1(\Omega)}^{2}.
\end{flalign*}
Combining all these estimates we arrive at our main results: Theorem \ref{convergence theorem 1} and Theorem \ref{convergence theorem 2}.

\section{The case of $d_b=0$}
The idea is to connect entropy dissipation with the missing $\Vert \delta_{B} \Vert^2_{\mathrm{L}^2(\Omega)}$ term in \eqref{new label d_a,d_c L2} so that we can apply Gr\"onwall inequality to have a sub-exponential decay. We begin by proving an uniform integrability estimate for $a$ and $c$ in a parabolic cylinder of unit height. A similar estimate was obtained for a degenerate four species model in \cite[Lemma 3.12, p.4343]{EMT20}.
\begin{Lem} \label{L^p with time}
Let  $p>N\geq 4$ and let $p'$ be its H\"older conjugate. Let $(a,b,c)$ be the solution to the degenerate system \eqref{eq:model 1} and let the nonzero diffusion coefficients $d_a,d_c$ satisfy the closeness condition \eqref{closeness condition imp}. Then there exists a $C_0>0$, depending only on the initial condition and the dimension $N$, such that
\[
\Vert a \Vert_{\mathrm L^{p}(\Omega_{\tau,\tau+1})}+\Vert c \Vert_{\mathrm L^{p}(\Omega_{\tau,\tau+1})} \leq C_0 \qquad \forall \tau>0.
\]
\end{Lem}
\begin{proof}
Define $d:=\frac{d_a+d_c}{2}$. We can rewrite the equations corresponding to concentrations $a$ and $c$ in \eqref{eq:model 1} as
\begin{equation} \label{rewrite a and c}
\left \{
\begin{aligned}
\partial_t a - d \Delta a & = c-ab+ (d_a -d) \Delta a,
\\
\partial_t c - d\Delta c & = ab-c +(d_c -d) \Delta c.
\end{aligned}
\right.   
\end{equation}
Let $\phi:[0,\infty)\rightarrow [0,1]$ be a smooth function such that $\phi(0)=0$ and 
\begin{align*}
\phi(x) & = 1 \qquad \mbox{ for }x\in[1,\infty)
\\
0 \le \phi'(x) & \le M \qquad \mbox{ for }x\in[0,\infty)
\end{align*}
for some constant $M>0$. For an arbitrary $\tau>0$, consider $\phi_\tau:[\tau,\infty)\to[0,1]$ defined as $\phi_{\tau}(s) := \phi(s-\tau)$ for $s\in[\tau,\infty)$. Then, the product $\phi_\tau(t)a(t,x)$ satisfies
\begin{equation} \nonumber
\left\{
\begin{aligned}
\partial_t \left( \phi_{\tau}a\right) - d \Delta \left( \phi_{\tau}a\right) & = a \partial_t \phi_{\tau} + \phi_{\tau}( c-ab+(d_a-d)\Delta a) \qquad \mbox{ in } \Omega_{\tau,\tau+2}\\
\phi_{\tau}\nabla a \cdot n & = 0 \qquad \qquad \qquad \qquad \  \qquad \qquad \qquad\ \ \ \ \mbox { on } \partial\Omega_{\tau,\tau+2}\\
\phi_{\tau}(\tau)a(\tau,x) & = 0 \qquad \qquad \qquad \qquad \qquad \qquad \qquad \ \ \ \ \  \mbox{ in }\Omega.
\end{aligned}
\right.
\end{equation}
Making the change of variable $t_1 = t - \tau$ in the above equation yields
\begin{equation}\label{eq:phi-a}
\left \{
\begin{aligned}
\partial_{t_1}(\phi_\tau(t_1+\tau)a(t_1+\tau,x)) & -d\Delta (\phi_\tau(t_1+\tau)a(t_1+\tau,x))
\\ 
= a(t_1+\tau,x)\partial_{t_1}\phi_\tau(t_1& +\tau)+\phi_\tau(t_1+\tau)(c-ab+(d_a-d)\Delta a) \qquad  \mbox{ in }\Omega_{0,2}
\\
\phi_\tau(t_1+\tau)\nabla a(t_1+\tau,x)\cdot n & =  0 \qquad   \qquad \quad \qquad \qquad  \ \qquad \qquad \qquad \qquad \mbox{ on } \partial\Omega_{0,2} 
\\
\phi_{\tau}(0+\tau)a(0+\tau,x) & = 0 \qquad \qquad \qquad \qquad \quad  \qquad \qquad  \qquad \qquad \ \ \mbox{ in } \Omega.
\end{aligned}
\right.
\end{equation}
Similarly, the concentration $c$ satisfies the following boundary value problem:
\begin{equation}\label{eq:phi-c}
\left \{
\begin{aligned}
\partial_{t_1}(\phi_\tau(t_1+\tau)c(t_1+\tau,x)) & -d\Delta (\phi_\tau(t_1+\tau)c(t_1+\tau,x))
\\ 
= c(t_1+\tau,x)\partial_{t_1}\phi_\tau(t_1& +\tau)+\phi_\tau(t_1+\tau)(ab-c+(d_c-d)\Delta c) \qquad  \mbox{ in }\Omega_{0,2} 
\\
\phi_\tau(t_1+\tau)\nabla c(t_1+\tau,x)\cdot n & =  0 \qquad  \quad \qquad \qquad \qquad \qquad  \qquad \qquad  \qquad \  \mbox{ on } \partial\Omega_{0,2} 
\\
\phi_{\tau}(0+\tau)c(0+\tau,x) & = 0 \qquad \qquad \qquad \qquad \quad \qquad  \qquad  \qquad \qquad \ \   \mbox{ in } \Omega. 
\end{aligned}
\right.
\end{equation}
Let $G_d$ denotes the Green's function associated with the operator $\partial_t-d\Delta$ with Neumann boundary condition. Then,  we can express the solutions to \eqref{eq:phi-a} and \eqref{eq:phi-c} as follows:
\begin{equation}\label{eq:represent-a}
\begin{aligned}
\phi_\tau(t_1+\tau)a(t_1+\tau,x) & = \int_{0}^{t_1}\int_{\Omega}G_d(t_1,s,x,y)a(s+\tau,y)\partial_{s}\phi_\tau(s+\tau)\, {\rm d}y\, {\rm d}s
\\
& + \int_{0}^{t_1}\int_{\Omega}G_d(t_1,s,x,y)\phi_\tau(s+\tau)(c-ab+(d_a-d)\Delta a)(s,y)\, {\rm d}y\, {\rm d}s
\end{aligned}
\end{equation}
and
\begin{equation}\label{eq:represent-c}
\begin{aligned}
\phi_\tau(t_1+\tau)c(t_1+\tau,x) & = \int_{0}^{t_1}\int_{\Omega}G_d(t_1,s,x,y)c(s+\tau,y)\partial_{s}\phi_\tau(s+\tau)\, {\rm d}y\, {\rm d}s
\\
& + \int_{0}^{t_1}\int_{\Omega}G_d(t_1,s,x,y)\phi_\tau(s+\tau)(ab-c+(d_c-d)\Delta c)(s,y)\, {\rm d}y\, {\rm d}s.
\end{aligned}
\end{equation}
Let us fix a non-negative $\theta \in \mathrm L^{p'}(\Omega_{0,2})$. Let $\psi$ be the solution to
\begin{equation}\label{eq:psi-evolve}
\left\{
\begin{aligned}
\partial_{t_1} \psi(t_1,x) +d\Delta \psi(t_1,x)& = -\theta(t_1,x) \qquad \mbox{in}\ \Omega_{0,2}
\\
\nabla \psi(t_1,x)\cdot n & =0 \qquad \ \ \qquad \ \  \mbox{on}\  (0,2) \times \partial \Omega
\\
\psi(2,x) & =0 \qquad \qquad \quad \  \mbox{in} \ \Omega.
\end{aligned}
\right.
\end{equation}
Applying the second order regularity estimate (see Theorem \ref{estimation 1} in Appendix \ref{sec:app} for further details) to the above equation yields
\begin{align}\label{second order psi}
\Vert \Delta \psi \Vert_{\mathrm L^{p'}(\Omega _{0,2})} \leq \frac{2 C_{SOR}( \Omega,N,p')}{d_a+d_c} \Vert \theta \Vert_{\mathrm L^{p'}(\Omega_{0,2})}.
\end{align}
Multiplying the expression \eqref{eq:represent-a} by $\theta$ and integrating over time and space yields
\begin{equation*}
\begin{aligned}
\int_{0}^{2}\int_{\Omega} & \phi_\tau(t_1+\tau) a(t_1+\tau,x) \theta(t_1,x)\, {\rm d}x\, {\rm d}t_{1}
\\
& = \int_{0}^{2}\int_{\Omega} \left(\int_{0}^{t_1}\int_{\Omega}G_d(t_1,s,x,y)a(s+\tau,y)\partial_{s}\phi_\tau(s+\tau)\, {\rm d}y\, {\rm d}s\right)\theta(t_1,x)\, {\rm d}x\, {\rm d}t_{1}
\\
& + \int_{0}^{2}\int_{\Omega} \left( \int_{0}^{t_1}\int_{\Omega}G_d(t_1,s,x,y)\phi_\tau(s+\tau)(c-ab+(d_a-d)\Delta a)(s,y)\, {\rm d}y\, {\rm d}s \right)\theta(t_1,x)\, {\rm d}x\, {\rm d}t_{1}.
\end{aligned}
\end{equation*}
Substituting for $\theta$ in terms of $\psi$, using \eqref{eq:psi-evolve}, in the second term on the right hand side of the above expression followed by integration by parts yields
\begin{align*}
\int_{0}^{2}\int_{\Omega} & \phi_\tau(t_1+\tau) a(t_1+\tau,x) \theta(t_1,x)\, {\rm d}x\, {\rm d}t_{1}
\\
& = \int_{0}^{2}\int_{\Omega} \left(\int_{0}^{t_1}\int_{\Omega}G_d(t_1,s,x,y)a(s+\tau,y)\partial_{s}\phi_\tau(s+\tau)\, {\rm d}y\, {\rm d}s\right)\theta(t_1,x)\, {\rm d}x\, {\rm d}t_{1}
\\
& + \int_{0}^{2}\int_{\Omega} \left( \int_{0}^{t_1}\int_{\Omega}\left(\partial_{t_1}-d\Delta\right)G_d(t_1,s,x,y)\left(\phi_\tau(s+\tau)(c-ab+(d_a-d)\Delta a)(s,y)\right)\, {\rm d}y\, {\rm d}s \right)\psi(t_1,x)\, {\rm d}x\, {\rm d}t_{1}
\\
& + \int_{0}^{2}\int_{\Omega} \left(\int_\Omega G_d(t_1,t_1,x,y)\phi_\tau(t_1+\tau)(c-ab+(d_a-d)\Delta a)(t_1,y)\, {\rm d}y\right)\psi(t_1,x)\, {\rm d}x\, {\rm d}t_{1}.
\end{align*}
Using the property of the Green's function, we get		
\begin{align*}
\int_{0}^{2}\int_{\Omega} & \phi_\tau(t_1+\tau) a(t_1+\tau,x) \theta(t_1,x)\, {\rm d}x\, {\rm d}t_{1}
\\
& = \int_{0}^{2}\int_{\Omega} \left(\int_{0}^{t_1}\int_{\Omega}G_d(t_1,s,x,y)a(s+\tau,y)\partial_{s}\phi_\tau(s+\tau)\, {\rm d}y\, {\rm d}s\right)\theta(t_1,x)\, {\rm d}x\, {\rm d}t_{1}
\\
& + \int_{0}^{2}\int_{\Omega} \phi_\tau(t_1+\tau)(c-ab+(d_a-d)\Delta a)(t_1,x)\psi(t_1,x)\, {\rm d}x\, {\rm d}t_{1}.
\end{align*}
A further integration by parts in the second term on the right hand side of the above expression yields
\begin{equation}\label{first intermediate}
\begin{aligned}
\int_{0}^{2}\int_{\Omega} & \phi_\tau(t_1+\tau) a(t_1+\tau,x) \theta(t_1,x)\, {\rm d}x\, {\rm d}t_{1}
\\
& \le \left\Vert \int_{0}^{t_1}\int_{\Omega}G_d(t_1,s,x,y)a(s+\tau,y)\partial_{s}\phi_\tau(s+\tau)\, {\rm d}y\, {\rm d}s \right\Vert_{\mathrm L^p(\Omega_{0,2})} \left\Vert \theta \right\Vert_{\mathrm L^{p'}(\Omega_{0,2})}
\\
& + \int_{0}^{2}\int_{\Omega} \phi_\tau(t_1+\tau)(\left(c-ab\right)\psi+(d_a-d)a\Delta \psi)(t_1,x)\, {\rm d}x\, {\rm d}t_{1},
\end{aligned}
\end{equation}
where we have also used H\"older inequality to bound the first term on the right hand side.

The following Green's function estimate is available from \cite{Morra83}\cite{ML15}: there exists a constant $K_1>0$, depending only on the domain, such that
\[
0\leq G_d(t_1,s,x,y) \leq \frac{K_1}{(4\pi( t_1-s))^\frac{N}{2}}e^{-\kappa \frac{\Vert x-y \Vert^2}{(t_1-s)}} =: g_{d}(t_1-s,x-y),
\]
for some constant $\kappa>0$ depending only on $\Omega$ and the diffusion coefficient $d$. Note that
\[
\left\Vert g_d\right\Vert_{\mathrm{L}^{z}(\Omega_{0,2})} \leq K_2 \qquad \forall z\in\left[1,1+\frac{N}{2}\right)
\]
for some constant $K_2>0$ depending on $z$. We, in particular choose $\displaystyle{z=1+\frac{1}{N}}$. Observe that there exists a $q<p$ such that
\[
1+\frac{1}{p} = \frac{1}{1+\frac{1}{N}} + \frac{1}{q}.
\]
Hence, applying the Young's convolution inequality we obtain
\begin{align*}
\left\Vert \int_{0}^{t_1}\int_{\Omega}G_d(t_1,s,x,y)a(s+\tau,y)\partial_{s}\phi_\tau(s+\tau)\, {\rm d}y\, {\rm d}s \right\Vert_{\mathrm L^p(\Omega_{0,2})}
& \le  \left\Vert g_d\right\Vert_{\mathrm L^{1+\frac{1}{n}}(\Omega_{0,2})} \left\Vert a(\cdot+\tau,\cdot)\partial_{s}\phi_\tau(\cdot+\tau) \right\Vert_{\mathrm L^q(\Omega_{0,2})}
\\
& \le K_2 \left\Vert a(\cdot+\tau,\cdot)\partial_{s}\phi_\tau(\cdot+\tau) \right\Vert_{\mathrm L^q(\Omega_{0,2})}.
\end{align*}
Using this estimate and the estimate of $\Delta \psi$ from \eqref{second order psi} in \eqref{first intermediate}, we arrive at
\begin{equation}\label{eq:phi-a-estimate}
\begin{aligned}
\int_{0}^{2}\int_{\Omega} & \phi_\tau(t_1+\tau) a(t_1+\tau,x) \theta(t_1,x)\, {\rm d}x\, {\rm d}t_{1}
\\
& \le K_2 \left\Vert a(\cdot+\tau,\cdot)\partial_{s}\phi_\tau(\cdot+\tau) \right\Vert_{\mathrm L^q(\Omega_{0,2})} \left\Vert \theta \right\Vert_{\mathrm L^{p'}(\Omega_{0,2})}
+ \int_{0}^{2}\int_{\Omega} \phi_\tau(t_1+\tau)\left(\left(c-ab\right)\psi\right)(t_1,x)\, {\rm d}x\, {\rm d}t_{1}
\\
& + C_{SOR}( \Omega,N,p') \frac{\vert d_a-d_c\vert}{d_a+d_c} \left\Vert \phi_{\tau}(\cdot+\tau)a(\cdot+\tau,\cdot) \right\Vert_{\mathrm L^{p}(\Omega_{0,2})} \Vert \theta \Vert_{\mathrm L^{p'}(\Omega_{0,2})}.
\end{aligned}
\end{equation}	
Performing a similar set of computations on the equation \eqref{eq:represent-c} for $\phi_\tau c$ yields
\begin{equation}\label{eq:phi-c-estimate}
\begin{aligned}
\int_{0}^{2}\int_{\Omega} & \phi_\tau(t_1+\tau) c(t_1+\tau,x) \theta(t_1,x)\, {\rm d}x\, {\rm d}t_{1}
\\
& \le K_2 \left\Vert c(\cdot+\tau,\cdot)\partial_{s}\phi_\tau(\cdot+\tau) \right\Vert_{\mathrm L^q(\Omega_{0,2})} \left\Vert \theta \right\Vert_{\mathrm L^{p'}(\Omega_{0,2})}
+ \int_{0}^{2}\int_{\Omega} \phi_\tau(t_1+\tau)\left(\left(ab-c\right)\psi\right)(t_1,x)\, {\rm d}x\, {\rm d}t_{1}
\\
& + C_{SOR}( \Omega,N,p') \frac{\vert d_a-d_c\vert}{d_a+d_c} \left\Vert \phi_{\tau}(\cdot+\tau)c(\cdot+\tau,\cdot) \right\Vert_{\mathrm L^{p}(\Omega_{0,2})} \Vert \theta \Vert_{\mathrm L^{p'}(\Omega_{0,2})}.
\end{aligned}
\end{equation}	
Adding the inequalities \eqref{eq:phi-a-estimate} and \eqref{eq:phi-c-estimate} and using positivity of $a$ and $c$ yields
\begin{equation*}
\begin{aligned}
\int_{0}^{2}\int_{\Omega} & \phi_\tau(t_1+\tau) \left(a+c\right)(t_1+\tau,x) \theta(t_1,x)\, {\rm d}x\, {\rm d}t_{1}
\\
& \le K_2 \left\Vert \left( a+ c\right)(\cdot+\tau,\cdot)\partial_{s}\phi_\tau(\cdot+\tau) \right\Vert_{\mathrm L^q(\Omega_{0,2})} \left\Vert \theta \right\Vert_{\mathrm L^{p'}(\Omega_{0,2})}
\\
& + C_{SOR}( \Omega,N,p') \frac{\vert d_a-d_c\vert}{d_a+d_c} \left\Vert \phi_{\tau}(\cdot+\tau)\left(a+c\right)(\cdot+\tau,\cdot) \right\Vert_{\mathrm L^{p}(\Omega_{0,2})} \Vert \theta \Vert_{\mathrm L^{p'}(\Omega_{0,2})}.
\end{aligned}
\end{equation*}		
As $1\leq q< p$, there exists $\alpha \in(0,1]$ such that
\[
\frac{1}{q}=\frac{1-\alpha}{p}+\frac{\alpha}{1}.
\]
Hence by interpolation, we have
\[
\Vert f \Vert_{\mathrm L^q} \leq \Vert f\Vert_{\mathrm L^p}^{1-\alpha}\Vert f \Vert_{\mathrm L^1}^{\alpha} \qquad \forall f\in \mathrm L^{q}\cap \mathrm L^{1}.
\]
Using the above interpolation inequality and a duality argument, we obtain
\begin{align*}\label{second intermediate}
\left\Vert \phi_{\tau}(\cdot+\tau)(a+c)(\cdot + \tau,\cdot) \right\Vert_{\mathrm L^p(\Omega_{0,2})}
\leq \frac{K_2M^\alpha M_1^{\alpha}\vert \Omega\vert^{\alpha}}{\left( 1 -  C_{SOR}(\Omega,N,p')\frac{\vert d_a-d_c\vert}{d_a+d_c}\right)} \left\Vert \left( a+ c\right)(\cdot+\tau,\cdot)\partial_{s}\phi_\tau(\cdot+\tau) \right\Vert_{\mathrm L^p(\Omega_{0,2})}^{1-\alpha}.
\end{align*}
To obtain bounds which are independent of $\tau$, let us take
\[
\texttt{C} := \frac{K_2M M_1^{\alpha}\vert \Omega\vert^{\alpha}}{\left( 1 -  C_{SOR}(\Omega,N,p')\frac{\vert d_a-d_c\vert}{d_a+d_c}\right)}
\quad
\mbox{ and }
\quad
\beta_n := \left\Vert a+c \right\Vert_{\mathrm L^p(\Omega_{n,n+1})} \, \mbox{ for }n\in\mathbb{N}\cup\{0\}.
\]		
As $\tau\geq0$ is arbitrary, we deduce from \eqref{second intermediate} that
\[
\beta_{n+1} \le \texttt{C}\, \beta_n^{1-\alpha} \qquad \mbox{ for }n\in\mathbb{N}.
\]
Consider the set
\[
\Lambda := \big\{n\in \mathbb{N} \quad \text{ such that }\quad \beta_n \leq \beta_{n+1}\big\}.
\]
Observe that  $\beta_n \leq \ {\texttt{C}}^{\frac{1}{\alpha}}$ for all $n\in \Lambda$. Hence we deduce that
\[
\beta_{n} \leq \max \big\{\beta_0,{\texttt{C}}^{\frac{1}{\alpha}} \big\}.
\]
Observe that the sum of concentrations $a+c$ satisfies the following differential equation:
\begin{equation*}
\left\{
\begin{aligned}
\partial_{t}(a+c)- \Delta(\mu(a+c)) & = 0  \qquad \qquad \qquad \qquad \ \mbox{in}\ \Omega_{T},
\\
\nabla_{x}(a+c)\cdot n & = 0  \qquad \qquad \qquad \qquad \  \mbox{on} \ [0,T]\times \partial\Omega,
\\
(a+c)(0,x) & = a_0+c_0\in {\mathrm L^p(\Omega)}  \qquad \mbox{in}\ \Omega,
\end{aligned}
\right.
\end{equation*}
where $\mu:\Omega_T\to\mathbb{R}$ defined as follows:
\[
\mu(t,x) := \left(\frac{d_a a+d_c c}{a+c}\right)(t,x) \qquad \mbox{ for }(t,x)\in\Omega_T.
\]
Observe that $\mu$ satisfies
\[
0 < \min\{d_a,d_c\} \le \mu(t,x) \le \max\{d_a,d_c\} \qquad \mbox{ for all }(t,x)\in\Omega_T.
\]
As the diffusion coefficients $d_a$ and $d_c$ satisfy the closeness condition \eqref{closeness condition imp}, employing the $p^{\textrm{th}}$ order integrability estimate \cite[Proposition 1.1, p.1186]{CDF14} (see Theorem \ref{PE} from the Appendix for the precise statement), we arrive at
\begin{align*}
\beta_0 = \Vert a+c \Vert_{\mathrm L^{p}(\Omega_{0,1})}
\leq 
\left(1+\max\{d_a,d_c\} \frac{\vert d_c-d_a\vert\, C^{PRC}_{\frac{d_a+d_c}{2},p'}}{2-\vert d_c-d_a\vert\, C^{PRC}_{\frac{d_a+d_c}{2},p'}} \right) \Vert  a_0+c_0 \Vert_{\mathrm L^p(\Omega)}. 
\end{align*}
Hence we deduce
\[
\beta_{n} 
\leq 
\max \left\{ \left(1+\max\{d_a,d_c\} \frac{\vert d_c-d_a\vert\, C^{PRC}_{\frac{d_a+d_c}{2},p'}}{2-\vert d_c-d_a\vert\, C^{PRC}_{\frac{d_a+d_c}{2},p'}} \right) \Vert  a_0+c_0 \Vert_{\mathrm L^p(\Omega)}, \, {\texttt{C}}^{\frac{1}{\alpha}}\right\}.
\]
Hence there exists a constant $C_0$, independent of $\tau$, such that
\[
\Vert a \Vert_{\mathrm L^{p}(\Omega_{\tau,\tau+1})}+\Vert c \Vert_{\mathrm L^{p}(\Omega_{\tau,\tau+1})} \leq C_0 \qquad \forall \tau>0.
\]
\end{proof}
The following lemma derives the key integrability estimate for the concentration $b$ which will play an important role in our analysis.
\begin{Lem} \label{N/2 estimate}
Let $p>N\geq 4$ and let $p'$ be its H\"older conjugate. Let $(a,b,c)$ be the solution to the degenrate system \eqref{eq:model 1} and let the nonzero diffusion coefficients $d_a,d_c$ satisfy the closeness condition \eqref{closeness condition imp}. Then there exists a constant $K_3>0$, depending only on the initial data and the dimension N, such that	
\[
\left\Vert b(t,\cdot)\right\Vert_{\mathrm L^{\frac{N}{2}}(\Omega)} \leq K_3\, (1+t)^{\frac{N-2}{N-1}} \qquad \forall t\geq 0.
\]
\end{Lem}
\begin{proof}
Recall that $b$ solves the equation
\[
\partial_t b = c - ab.
\]
Non-negativity of $a$ and $b$ implies that
\[
\partial_t b \le c.
\]
Integrating over $(0,t)$ yields
\[
b(t,x) \le b_0(x) + \int_0^t c(s,x)\, {\rm d}s.
\]
Raising it to the power $N$ and employing Jensen's inequality, we obtain
\[
\left( b(t,x)\right)^N 
\le 2^{N-1} \left( \left(b_0(x)\right)^N + \left( \int_0^t c(s,x)\, {\rm d}s\right)^N \right)
\le 2^{N-1}(1+t)^{N-1} \left( \left(b_0(x)\right)^N + \int_0^t \left( c(s,x)\right)^N\, {\rm d}s \right).
\]
Integrating the above inequality in the $x$ variable over $\Omega$ yields
\begin{align}\label{eq:inter-b^N}
\left\Vert b(t,\cdot)\right\Vert_{\mathrm L^N(\Omega)}^N \le 2^{N-1}(1+t)^{N-1} \Big( \left\Vert b_0\right\Vert_{\mathrm L^N(\Omega)}^N + \left\Vert c \right\Vert_{\mathrm L^N(\Omega_t)}^N \Big).
\end{align}
According to Lemma \ref{L^p with time}, for any $p>N$ and for any $\tau\ge0$, we have
\[
\left\Vert c\right\Vert_{\mathrm L^p(\Omega_{\tau,\tau+1})} \le C_0,
\]
where the constant $C_0>0$ depends only on the dimension $N$, the domain $\Omega$ and the initial data. H\"older inequality yields
\[
\left\Vert c\right\Vert_{\mathrm L^N(\Omega_{\tau,\tau+1})} \le \left\Vert c\right\Vert_{\mathrm L^p(\Omega_{\tau,\tau+1})}\, \left\vert \Omega\right\vert^{\frac{p-N}{p}}.
\]
Hence we deduce that for any $t>0$,
\[
\left\Vert c\right\Vert_{\mathrm L^N(\Omega_t)}^N  \le C_0^N\, \left\vert \Omega\right\vert^{\frac{N(p-N)}{p}}\, (1+t).
\]
Using this in \eqref{eq:inter-b^N} yields
\[
\left\Vert b(t,\cdot)\right\Vert_{\mathrm L^N(\Omega)}^N \le 2^{N-1}(1+t)^{N-1} \left( \left\Vert b_0\right\Vert_{\mathrm L^\infty(\Omega)}^N \, \left\vert \Omega \right\vert + C_0^N\, \left\vert \Omega\right\vert^{\frac{N(p-N)}{p}}\, (1+t) \right).
\]
Define a constant $K_3>0$ as follows:
\[
K_3^{\frac{N(N-1)}{N-2}} := 2^{N-1} \left(M_{2}\vert \Omega\vert\right)^{\frac{N}{N-2}} \left( \left\Vert b_0\right\Vert_{\mathrm L^\infty(\Omega)}^N \, \left\vert \Omega \right\vert + C_0^N\, \left\vert \Omega\right\vert^{\frac{N(p-N)}{p}} \right).
\]
Thus we arrive at
\[
\left\Vert b(t,\cdot)\right\Vert_{\mathrm L^N(\Omega)} \le K_3^{\frac{N-1}{N-2}} \left(M_{2}\vert \Omega\vert\right)^{-\frac{1}{N-2}} (1+t).
\]
Note that
\[
\frac{2}{N} = \frac{\alpha}{N} + \frac{1-\alpha}{1} \qquad \mbox{ with }\, \alpha=\frac{N-2}{N-1}.
\]
Hence by interpolation, we have
\[
\left\Vert b(t,\cdot) \right\Vert_{\mathrm L^{\frac{N}{2}}(\Omega)} \leq \left\Vert b(t,\cdot) \right\Vert_{\mathrm L^N(\Omega)}^{\frac{N-2}{N-1}} \, \left\Vert b(t,\cdot) \right\Vert_{\mathrm L^1(\Omega)}^{\frac{1}{N-1}}.
\]
This yields the following estimate:
\[
\left\Vert b(t,\cdot) \right\Vert_{\mathrm L^{\frac{N}{2}}(\Omega)} \leq K_3\, (1+t)^{\frac{N-2}{N-1}}.
\]
\end{proof}
Unlike the case of dimension $N\ge4$, we are able to obtain $\mathrm L^\frac32(\Omega)$-norm estimates on $b(t,\cdot)$ without any closeness assumption on the diffusion coefficients $d_a$ and $d_c$ in the case of dimensions $N\le3$. Next, we prove such an estimate and more in dimension three followed by a similar result in dimensions one and two.	
\begin{Prop}\label{N=3 main}
Let $(a,b,c)$ be the solution to the degenerate system \eqref{eq:model 1} in dimension $N=3$. Then, there exist constants $\hat{K},K_c>0$ and $\mu_c\in\mathbb{N}$, independent of time, such that
\begin{align*}
\left\Vert b(t,\cdot)\right\Vert_{\mathrm L^\frac32(\Omega)} & \le \hat{K}(1+t)^{\frac{5}{6}} \qquad \mbox{ for } \, t\geq 0,
\\
\left\Vert a\right\Vert_{\mathrm L^\frac72(\Omega_{\tau,\tau+1})} & \le K_c \qquad \mbox{ for } \, \tau\geq 0,
\\
\left\Vert c\right\Vert_{\mathrm L^\frac72(\Omega_{\tau,\tau+1})} & \le K_c(1+t)^{\mu_c} \qquad \mbox{ for } \, \tau\geq 0.
\end{align*}
\end{Prop}
\begin{proof}
Recall that the solution $(a,b,c)$ to the degenerate system \eqref{eq:model 1} satisfies the following estimate:
\[
\int_0^t \int_\Omega \left( 4 d_a \left\vert \nabla \sqrt{a} \right\vert^2  + 4 d_c \left\vert \nabla \sqrt{c} \right\vert^2 + (ab-c)\ln\left(\frac{ab}{c}\right)\right)\, {\rm d}x\, {\rm d}s
\leq E(a_0,b_0,c_0) - E(a, b, c)
\leq E(a_0,b_0,c_0),
\]
where $E$ is the entropy functional defined in \eqref{entropy} and $(a_\infty, b_\infty, c_\infty)$ is the homogeneous equilibrium state defined by \eqref{eq:equi-state-1}-\eqref{eq:equi-state-2}. The above estimate along with the mass conservation properties \eqref{mass conserve 1}-\eqref{mass conserve 2} results in
\begin{align}\label{eq:apriori-1}
\int_{\tau}^{\tau+2} \left\Vert \sqrt{a}(t,\cdot) \right\Vert_{\mathrm W^{1,2}(\Omega)}^2\, {\rm d}t \leq K_4
\quad
\mbox{ and }
\quad
\int_{\tau}^{\tau+2} \left\Vert \sqrt{c}(t,\cdot) \right\Vert_{\mathrm W^{1,2}(\Omega)}^2\, {\rm d}t \leq K_4 
\qquad \mbox{ for any }\tau\ge0,
\end{align}
where $K_4>0$ is the following constant
\[
K_4 := \frac{E(a_0,b_0,c_0)}{4d_a} +  \frac{E(a_0,b_0,c_0)}{4d_c} + 2M_1 \left\vert \Omega\right\vert + 2M_2 \left\vert \Omega\right\vert.
\]
Gagliardo-Nirenberg inequality says that there exist constants $K_5=K_5(\Omega,N)$ and $K_6=K_6(\Omega,N,r,q)$ where $N$ is the dimension and $1\le p,q,r\le+\infty$ such that for nonnegative integers $j<m$,
\[
\left\Vert D^j u \right\Vert_{\mathrm L^p(\Omega)} \leq K_5 \left\Vert D^m u \right\Vert_{\mathrm L^r(\Omega)}^{\alpha} \left\Vert u \right\Vert_{\mathrm L^q(\Omega)}^{1-\alpha} + K_6 \left\Vert u \right\Vert_{\mathrm L^q(\Omega)}
\quad
\mbox{ for any }u\in\mathrm L^q(\Omega) \mbox{ satisfying }D^m u\in \mathrm L^r(\Omega).
\]		
Furthermore, the various parameters in the above inequality satisfy
\[
\frac{1}{p}=\frac{j}{N}+\left(\frac{1}{r}-\frac{m}{N}\right)\alpha\ + \frac{1-\alpha}{q}
\qquad
\mbox{ and }
\qquad
\frac{j}{N}\leq \alpha \leq 1.
\]
Applying the Gagliardo-Nirenberg inequality for $\sqrt{a}$ by taking $r=2, m=1, j=0, N=3, q=2$ and $\alpha=\frac{3}{5}$ yields
\[
\left\Vert \sqrt{a}(t,\cdot) \right\Vert_{\mathrm L^\frac{10}{3}(\Omega)} 
\le K_5 \left\Vert \nabla \sqrt{a}(t,\cdot) \right\Vert_{\mathrm L^2(\Omega)}^\frac35\, \left\Vert \sqrt{a}(t,\cdot) \right\Vert_{\mathrm L^2(\Omega)}^\frac25
+ K_6 \left\Vert \sqrt{a}(t,\cdot) \right\Vert_{\mathrm L^2(\Omega)}.
\]
Raising it to the power $\frac{10}{3}$ yields
\begin{align*}
\left\Vert \sqrt{a}(t,\cdot) \right\Vert_{\mathrm L^\frac{10}{3}(\Omega)}^\frac{10}{3}
& \le 2^\frac73 K_5^\frac{10}{3} \left\Vert \nabla \sqrt{a}(t,\cdot) \right\Vert_{\mathrm L^2(\Omega)}^2\, \left\Vert \sqrt{a}(t,\cdot) \right\Vert_{\mathrm L^2(\Omega)}^\frac43
+ 2^\frac73 K_6^\frac{10}{3} \left\Vert \sqrt{a}(t,\cdot) \right\Vert_{\mathrm L^2(\Omega)}^\frac{10}{3}
\\
& = 2^\frac73 K_5^\frac{10}{3} \left\Vert \nabla \sqrt{a}(t,\cdot) \right\Vert_{\mathrm L^2(\Omega)}^2\, M_1^\frac43 \, \left\vert \Omega\right\vert^\frac43
+ 2^\frac73 K_6^\frac{10}{3} \, M_1^\frac{10}{3} \, \left\vert \Omega\right\vert^\frac{10}{3}.
\end{align*}
Take an arbitrary $\tau\ge0$ and integrate the above inequality from $\tau$ to $\tau+2$ in the $t$ variable to get
\[
\int_\tau^{\tau+2}\left\Vert \sqrt{a}(t,\cdot) \right\Vert_{\mathrm L^\frac{10}{3}(\Omega)}^\frac{10}{3}\, {\rm d}t
\le 2^\frac73 K_5^\frac{10}{3}\, K_4\, M_1^\frac43 \, \left\vert \Omega\right\vert^\frac43
+ 2^\frac{10}{3} K_6^\frac{10}{3} \, M_1^\frac{10}{3} \, \left\vert \Omega\right\vert^\frac{10}{3},
\]
where we have used the apriori bound \eqref{eq:apriori-1}. Similarly, we can obtain the following estimate for $\sqrt{c}$:
\[
\int_\tau^{\tau+2}\left\Vert \sqrt{c}(t,\cdot) \right\Vert_{\mathrm L^\frac{10}{3}(\Omega)}^\frac{10}{3}\, {\rm d}t
\le 2^\frac73 K_5^\frac{10}{3}\, K_4\, M_1^\frac43 \, \left\vert \Omega\right\vert^\frac43
+ 2^\frac{10}{3} K_6^\frac{10}{3} \, M_1^\frac{10}{3} \, \left\vert \Omega\right\vert^\frac{10}{3}.
\]
Define a constant $K_7>0$ as follows:
\[
K_7 := \left( 2^\frac73 K_5^\frac{10}{3}\, K_4\, M_1^\frac43 \, \left\vert \Omega\right\vert^\frac43
+ 2^\frac{10}{3} K_6^\frac{10}{3} \, M_1^\frac{10}{3} \, \left\vert \Omega\right\vert^\frac{10}{3}\right)^\frac35.
\]
We have thus arrived at the following integrability estimates:
\begin{align}\label{first Lp for c}
\left\Vert a \right\Vert_{\mathrm L^{\frac{5}{3}}(\Omega_{\tau,\tau+2})} \leq K_7
\quad
\mbox{ and }
\quad
\left\Vert c \right\Vert_{\mathrm L^{\frac{5}{3}}(\Omega_{\tau,\tau+2})} \leq K_7
\quad
\mbox{ with }\tau\ge0.
\end{align}
Recall that non-negativity of $a,b$ implies $\partial_t b \leq c$ which in turn implies that
\[
b(t,x) \le b_0(x) + \int_0^t c(s,x)\, {\rm d}s.
\]
Raising the above inequality to the power $\frac53$ yields
\[
\left(b(t,x)\right)^\frac53 \le 2^\frac23 (1+t)^\frac23 \left( \left(b_0(x)\right)^\frac53 + \int_0^t \left(c(s,x)\right)^\frac53\, {\rm d}s\right).
\]
Integrating the above inequality over $\Omega$ yields
\begin{align*}
\int_\Omega \left(b(t,x)\right)^\frac53\, {\rm d}x & \le 2^\frac23 (1+t)^\frac23 \left( \int_\Omega \left(b_0(x)\right)^\frac53\, {\rm d}x + \int_0^t \int_\Omega \left(c(s,x)\right)^\frac53\, {\rm d}x\, {\rm d}s\right)
\\
& \le 2^\frac23 (1+t)^\frac23 \left( \left\Vert b_0\right\Vert_{\mathrm L^\infty(\Omega)}^\frac53\, \left\vert \Omega \right\vert + \sum_{\tau=0}^{1+\lfloor t\rfloor}\left\Vert c\right\Vert_{\mathrm L^\frac53(\Omega_{\tau,\tau+1})}^\frac53\right)
\\
& \le 2^\frac23 (1+t)^\frac23 \left( \left\Vert b_0\right\Vert_{\mathrm L^\infty(\Omega)}^\frac53\, \left\vert \Omega \right\vert + (1+t) K_7^\frac53\right)
\\
& \le 2^\frac23 (1+t)^\frac53\left( \left\Vert b_0\right\Vert_{\mathrm L^\infty(\Omega)}^\frac53\, \left\vert \Omega \right\vert + K_7^\frac53\right).
\end{align*}
Define a constant $K_8>0$ as follows:
\[
K_8 := \left\Vert b_0\right\Vert_{\mathrm L^\infty(\Omega)}^\frac53\, \left\vert \Omega \right\vert + K_7^\frac53.
\]
We have thus arrived at the following integrability estimate:
\begin{align}\label{5/3}
\left\Vert b(t,\cdot)\right\Vert_{\mathrm L^\frac53(\Omega)}^\frac53 \le K_8 (1+t)^\frac53 \qquad \mbox{ for }t\ge0.
\end{align}
Note that 
\[
\frac{2}{3}= \frac{3\alpha}{5}+\frac{1-\alpha}{1} \qquad \mbox{ with }\, \alpha=\frac{5}{6}.
\]
Hence by interpolation, we have
\begin{align*}
\left\Vert b(t,\cdot)\right\Vert_{\mathrm L^\frac32(\Omega)}
& \le \left\Vert b(t,\cdot)\right\Vert_{\mathrm L^\frac53(\Omega)}^\frac56 \, \left\Vert b(t,\cdot)\right\Vert_{\mathrm L^1(\Omega)}^\frac16
\le K_8^\frac12\, M_2^\frac16 \left\vert \Omega \right\vert^\frac16 (1+t)^\frac56,
\end{align*}
thanks to the estimate \eqref{5/3} and the mass conservation property \eqref{mass conserve 2}.

The species $a$ is a positive subsolution of the following equation in the time interval $(0,2)$.
\begin{equation}\label{eq:tilde a}
	\left \{
	\begin{aligned}
		\partial_{t}a(t,x)-d_a\Delta a(t,x)
		\leq &c  \qquad \qquad \qquad  \mbox{ in }\Omega_{0,2}
		\\
		\nabla a(t,x)\cdot n=&  0 \qquad   \qquad \quad   \ \  \mbox{ on } \partial\Omega_{0,2} 
		\\
		a(0,x) =& a_0 \qquad   \quad  \qquad \ \ \mbox{ in } \Omega.
	\end{aligned}
	\right.
\end{equation}
The solution corresponding to the equation \eqref{eq:tilde a} can be expressed as:
\[
a\leq \int_{\Omega}G_{d_a}(t,0,x,y)a_0(y) \rm{d}y+ \int_{0}^{t}\int_{\Omega} G_{d_a}(t,s,x,y)c(s,y)\rm{d}y\rm{d}s.
\]
where $G_{d_a}$ denotes the Green's function associated with the operator $\partial_t-d_a\Delta$ with Neumann boundary condition. We use the fact that $\displaystyle{a_0(y)\in\mathrm{L}^{\infty}(\Omega)}$ and $\displaystyle{\int_{\Omega}G_{d_a}(t,0,x,y) \rm{d}y \leq 1}$ for all $t\in(0,2)$. It yields
\[
\left\vert \int_{\Omega}G_{d_a}(t,0,x,y)a_0(y) \rm{d}y \right\vert \leq \Vert a_0\Vert_{\mathrm{L}^{\infty}(\Omega)}.
\]
We use the following Green's function estimate is from \cite{Morra83}\cite{ML15}: there exists a constant $K_1>0$, depending only on the domain, such that
\[
0\leq G_{d_a}(t_1,s,x,y) \leq \frac{K_1}{(4\pi( t_1-s))^\frac{N}{2}}e^{-\kappa \frac{\Vert x-y \Vert^2}{(t_1-s)}} =: g_{d_a}(t_1-s,x-y),
\]
for some constant $\kappa>0$ depending only on $\Omega$ and the diffusion coefficient $d$. Note that
\[
\left\Vert g_d\right\Vert_{\mathrm{L}^{z}(\Omega_{0,2})} \leq K_2 \qquad \forall z\in\left[1,1+\frac{N}{2}\right)
\]
for some constant positive $K_2$ depending on $z$. We, in particular choose $\frac{45}{28}$, which is an admissible number for $N=3$. As we have $\displaystyle{1+\frac{2}{9}=\frac{28}{45}+\frac{3}{5}}$, Young's convolution inequality yields 
	\[
	\left\Vert \int_{0}^{t}\int_{\Omega} G_{d_a}(t,s,x,y)c(s,y)\rm{d}y\rm{d}s  \right\Vert_{\mathrm{L}^{\frac{9}{2}}(\Omega_{0,2})} \leq \Vert g_d\Vert_{\mathrm{L}^{\frac{45}{28}}(\Omega_{0,2})} \Vert c\Vert_{\mathrm{L}^{\frac{5}{3}}(\Omega_{0,2})}\leq K_2 \Vert c\Vert_{\mathrm{L}^{\frac{5}{3}}(\Omega_{0,2})}.
	\]
Exploiting the non-negativity of the species $a$ and using the estimate from \eqref{first Lp for c}, we obtain
\[
\Vert a\Vert_{\mathrm{L}^{4.5}(\Omega_{0,2})}\leq \Vert a_0\Vert_{\mathrm{L}^{\infty}(\Omega)} 2^{\frac{2}{9}}\vert \Omega\vert^{\frac{2}{9}}+K_2K_7.
\]
Consider  $0\leq\Theta \in C_c^{\infty}(\Omega_{0,2})$(space of all compactly supported smooth function) satisfies
	\begin{equation*}
		\left \{
		\begin{aligned}
-[\partial_t \phi+d_c \Delta\phi]=& \Theta \qquad \ \ \ \mbox{in}\ \Omega_{0,2}\\
\nabla \phi.n=& 0 \qquad \mbox{on}\ (0,2)\times \partial\Omega\\
\phi(2)=&0 \qquad \ \ \ \ \mbox{in}\ \Omega.
	\end{aligned}
\right .
\end{equation*}
	
We have  $\phi \geq 0$,  and the following estimates for a constant $C_{q,d_c}>0$ ( $q\in(1,\infty)$, arbitrary)\cite{Pie2010}\cite{Amann1985}
	
	\[
\Vert\phi_t \Vert_{\mathrm{L}^q(\Omega_{0,2})}+\Vert\Delta \phi \Vert_{\mathrm{L}^q(\Omega_{0,2})}+\sup\limits_{s\in[0,2]}\Vert\phi(s)\Vert_{\mathrm{L}^q(\Omega)}+\Vert\phi\Vert_{\mathrm{L}^q((0,2)\times \partial\Omega)}\leq C_{q,d_c}\Vert \Theta \Vert_{\mathrm{L}^q(\Omega_{0,2})}.
	\]
We derive the following estimate of $c$, through integration by parts:
	\begin{align*}
\int_{\Omega_{0,2}} c \Theta =& \int_{\Omega} c_0 \phi(0)-\int_{\Omega_{0,2}} (\partial_t -d_c \Delta)c \phi\\
=&\int_{\Omega} c_0 \phi(0)+ \int_{\Omega_{0,2}} (\partial_t -d_a\Delta)a \phi =\int_{\Omega} (c_0+a_0) \phi(0)+ \int_{\Omega_{0,2}} a \partial_t \phi +\int_{\Omega_{0,2}} a d_a \Delta \phi\\
\leq & \Vert a_0+c_0\Vert_{\mathrm{L}^p(\Omega)} \Vert \phi(0)\Vert_{\mathrm{L}^q(\Omega)}+\Vert a\Vert_{L^p(\Omega_{0,2})} \Vert \partial_t \phi \Vert_{L^q(\Omega_{0,2})}+d_a \Vert a\Vert_{\mathrm{L}^p(\Omega_{0,2})} \Vert \Delta \phi \Vert_{\mathrm{L}^q(\Omega_{0,2})}.
	\end{align*}
Choose $q=\frac{4.5}{4.5-1}$, i.e., the H\"older conjugate of $4.5$, duality estimate yields
	\[
\Vert c\Vert_{\mathrm{L}^{4.5}(\Omega_{0,2})}\leq \Vert a_0+c_0\Vert_{\mathrm{L}^{4.5}(\Omega)}+2C_{q,d_c}\Vert a\Vert_{\mathrm{L}^{4.5}(\Omega_{0,2})}.
	\]
	
Choose $K_{0,1}:=\max \left \{\Vert a_0\Vert_{\mathrm{L}^{\infty}(\Omega)} 2^{\frac{2}{9}}\vert \Omega\vert^{\frac{2}{9}}+K_2K_7+\Vert a_0+c_0\Vert_{L^{4.5}(\Omega)} \right\}$, we have the following estimate
	\[
\Vert a \Vert_{\mathrm{L}^{4.5}(\Omega_{0,1})}+ \Vert c \Vert_{\mathrm{L}^{4.5}(\Omega_{0,1})} \leq K_{0,1}.
	\]	
Let $\phi:[0,\infty)\rightarrow [0,1]$ be a smooth function such that $\phi(0)=0$ and 
\begin{align*}
\phi(x) & = 1 \qquad \mbox{ for }x\in[1,\infty)
\\
0 \le \phi'(x) & \le M \qquad \mbox{ for }x\in[0,\infty)
\end{align*}
for some constant $M>0$. For an arbitrary $\tau>0$, consider $\phi_\tau:[\tau,\infty)\to[0,1]$ defined as $\phi_{\tau}(s) := \phi(s-\tau)$ for $s\in[\tau,\infty)$. Then, the product $\phi_\tau(t)a(t,x)$ satisfies
\begin{equation} \nonumber
\left\{
\begin{aligned}
\partial_t \left( \phi_{\tau}a\right) - d_a \Delta \left( \phi_{\tau}a\right) & = a \partial_t \phi_{\tau} + \phi_{\tau}( c-ab) \qquad \ \ \mbox{ in } \Omega_{\tau,\tau+2}\\
\phi_{\tau}\nabla a \cdot n & = 0 \qquad \qquad \qquad \qquad \qquad\ \mbox { on } \partial\Omega_{\tau,\tau+2}\\
\phi_{\tau}(\tau)a(\tau,x) & = 0 \qquad \qquad \qquad \qquad \qquad \ \  \mbox{ in }\Omega.
\end{aligned}
\right.
\end{equation}
Let $\zeta(t,x)$ be the solution to the following initial boundary value problem:
\begin{equation*}
\left\{
\begin{aligned}
\partial_t \zeta - d_a \Delta \zeta & = a \partial_t \phi_{\tau} + \phi_{\tau}c \qquad \qquad   \mbox{in} \ \Omega_{\tau,\tau+2}
\\
\nabla \zeta \cdot n & = 0 \qquad \qquad \qquad \qquad \mbox{on}\  \partial\Omega_{\tau,\tau+2}
\\
\zeta(\tau,x) & = 0 \qquad \qquad \qquad \qquad \ \ \mbox{in} \ \Omega.
\end{aligned}
\right.
\end{equation*}
By employing the maximum principle for the heat equation and exploiting the positivity of $a,b$, we deduce
\begin{align}\label{eq:pointwise-bd}
\phi_\tau(t,x) a(t,x) \le \zeta(t.x) \qquad \mbox{ for }(t,x)\in\Omega_{\tau,\tau+2}.
\end{align}
Employing the integrability estimate for $\zeta$ available from \cite[Lemma 3.3]{CDF14} (see Theorem \ref{estimation 1} from the Appendix  for the precise statement), we deduce that
\[
\left\Vert \zeta \right\Vert_{\mathrm L^s(\Omega_{\tau,\tau+2})} 
\le C_{IE}(\Omega,d_a,s) \left\Vert a \partial_t \phi_{\tau} + \phi_{\tau}c \right\Vert_{\mathrm L^\frac53(\Omega_{\tau,\tau+2})}
\le 2 C_{IE}(\Omega,d_a,s) (1+M) K_7
\qquad \mbox{ for any }s<5,
\]
where the second inequality is a consequence of the properties of smooth function $\phi$ and the integrability estimates on $a$ and $c$ from \eqref{first Lp for c}. Thanks to the pointwise bound \eqref{eq:pointwise-bd}, we in particular (taking $s=\frac92$) have
\[
\left\Vert a \right\Vert_{\mathrm L^\frac92(\Omega_{\tau+1,\tau+2})}
\le 
\left\Vert \zeta \right\Vert_{\mathrm L^\frac92(\Omega_{\tau,\tau+2})}
\le 2 C_{IE}(\Omega,d_a) (1+M) K_7
\qquad \mbox{ for }\tau\ge0.
\]
Hence we deduce for $\tau\ge0$,
\begin{align}\label{first Lp for a}
\left\Vert a \right\Vert_{\mathrm L^\frac92(\Omega_{\tau,\tau+1})}
\le \left\Vert a \right\Vert_{\mathrm L^\frac92(\Omega_{0,1})}
+ 2 C_{IE}(\Omega,d_a) (1+M) K_7 
\le K_{0,1} + 2 C_{IE}(\Omega,d_a) (1+M) K_7 
=: K_9.
\end{align}
Furthermore, integrating the estimate \eqref{5/3} in the time variable from $\tau$ to $\tau+1$ helps us obtain
\[
\left\Vert b\right\Vert_{\mathrm L^\frac53(\Omega_{\tau,\tau+1})} \le 2 K_8^\frac35 (1+\tau) =: K_{10} (1+\tau) \qquad \mbox{ for }\tau\ge0.
\]
Observe that an application of the H\"older inequality yields
\[
\left\Vert ab\right\Vert_{\mathrm L^{\frac{45}{37}}(\Omega_{\tau,\tau+1})} 
\le \left\Vert a \right\Vert_{\mathrm L^\frac92(\Omega_{\tau,\tau+1})}^\frac{2}{9} \, \left\Vert b\right\Vert_{\mathrm L^\frac53(\Omega_{\tau,\tau+1})}^\frac{3}{5}
\le K_9^{\frac{2}{9}} K_{10}^\frac35 (1+\tau)^\frac35
\]
Using the smooth function $\phi$ defined earlier, remark that the product $\phi_\tau(t)c(t,x)$ satisfies
\begin{equation*}
\left\{
\begin{aligned}
\partial_t \left(\phi_{\tau}c \right) - d_c \Delta \left( \phi_{\tau} c\right) = & c \partial_t \phi_{\tau} +\phi_{\tau}( ab-c) \qquad \mbox{ in } \Omega_{\tau,\tau+2}\\
\phi_{\tau}\nabla c \cdot n & = 0 \qquad \qquad \qquad \qquad\ \mbox { on } \partial\Omega_{\tau,\tau+2}\\
\phi_{\tau}(\tau)c(\tau,x) & = 0 \qquad \qquad \qquad \qquad \   \mbox{ in }\Omega.
\end{aligned}
\right.
\end{equation*}
where $\phi_{\tau}(s) := \phi(s-\tau)$ for $s\in[\tau,\infty)$ with $\tau>0$. Thanks to the positivity of $c$, employing the maximum principle for the heat equation and the integrability estimate from \cite[Lemma 3.3]{CDF14} (see Theorem \ref{estimation 1} from the Appendix for the precise statement), we arrive at
\begin{align*}
\left\Vert \phi_{\tau}c \right\Vert_{\mathrm L^s(\Omega_{\tau,\tau+2})}
& \le C_{IE}(\Omega,d_c,s) \left\Vert c \partial_t \phi_{\tau} +\phi_{\tau} ab \right\Vert_{\mathrm L^\frac{45}{37}(\Omega_{\tau,\tau+2})}
\\
& \le C_{IE} \left\{ M \left\Vert c\right\Vert_{\mathrm L^{\frac{45}{37}} (\Omega_{\tau,\tau+2})} + \left\Vert ab\right\Vert_{\mathrm L^{\frac{45}{37}} (\Omega_{\tau,\tau+2})} \right\}
\\
& \le C_{IE} \left\{ M \left\Vert c\right\Vert^\frac59_{\mathrm L^1 (\Omega_{\tau,\tau+2})} \left\Vert c\right\Vert^\frac49_{\mathrm L^\frac53 (\Omega_{\tau,\tau+2})} + \left\Vert ab\right\Vert_{\mathrm L^{\frac{45}{37}} (\Omega_{\tau,\tau+2})} \right\}
\\
& \le C_{IE} \left\{ M M_1^\frac59 K_7^\frac49 + 2K_9^{\frac{2}{9}} K_{10}^\frac35 (1+\tau)^\frac35 \right\}
\end{align*}
for any $s<\frac{(5)\frac{45}{37}}{(5-\frac{90}{37})}$.  Taking $s=\frac{9}{4}$ leads to
\[
\left\Vert \phi_{\tau}c \right\Vert_{\mathrm L^\frac{9}{4}(\Omega_{\tau,\tau+2})} \le K_{11} (1+\tau)^{\frac35}
\]
where $K_{11}:= C_{IE}(\Omega,d_c) \max\left\{M M_1^\frac59 K_7^\frac49, 2K_9^{\frac{2}{9}} K_{10}^\frac35\right\}$. Hence we deduce
\[
\left\Vert c\right\Vert_{\mathrm L^\frac{9}{4}(\Omega_{\tau,\tau+1})} 
\le \left\Vert c\right\Vert_{\mathrm L^\frac{9}{4}(\Omega_{0,1})} + K_{11} (1+\tau)^{\frac35}
\le K_{0,1} \left\vert \Omega\right\vert^\frac12 + K_{11} (1+\tau)^{\frac35} \qquad \mbox{ for }\tau\ge0.
\]
Again, exploiting the relation $\partial_t b \le c$, we arrive at
\begin{align*}
\left\Vert b(t,\cdot) \right\Vert^\frac{9}{4}_{\mathrm L^\frac94(\Omega)} 
& \le 2^\frac54 (1+t)^\frac54 \left( \left\Vert b_0\right\Vert_{\mathrm L^\infty(\Omega)}^\frac94\, \left\vert \Omega \right\vert + (1+t)^\frac{47}{20} K_{12}^\frac94\right)
\le K_{13} (1+t)^\frac{18}{5},
\end{align*}
thus deducing
\[
\left\Vert b \right\Vert_{\mathrm L^\frac94(\Omega_{\tau,\tau+1})} \le \left(\frac{5K_{13}}{23}\right)^\frac{4}{9} (1+\tau)^{\frac{92}{45}} \qquad \mbox{ for }\tau\ge0.
\]
H\"older inequality results in 
\[
\left\Vert ab\right\Vert_{\mathrm L^\frac32(\Omega_{\tau,\tau+1})} 
\le \left\Vert a\right\Vert_{\mathrm L^\frac92(\Omega_{\tau,\tau+1})} \left\Vert b\right\Vert_{\mathrm L^\frac94(\Omega_{\tau,\tau+1})}
\le K_9 \left(\frac{5K_{13}}{23}\right)^\frac{4}{9} (1+\tau)^{\frac{92}{45}},
\]
where we have used the bounds obtained earlier. As $\frac72 < \frac{(5)\frac32}{5-3}$, employing the maximum principle for the heat equation and the integrability estimate from \cite[Lemma 3.3]{CDF14} (see Theorem \ref{estimation 1}  for the precise statement) as before, we get
\[
\left\Vert c\right\Vert_{\mathrm L^\frac72(\Omega_{\tau,\tau+1})} \le K_{14} (1+\tau)^{\frac{92}{45}} \qquad \mbox{ for }\tau\ge0.
\]
\end{proof}

	\vspace{.3cm}
	
	In our above proof Gagliardo-Nirenberg inequality played an important role. We will use the same technique to obtain  $\mathrm{L}^{\frac{3}{2}}(\Omega)$ integral estimation of  $b$  for dimension $N=1,2$.
	
	\vspace{.3cm}
	\begin{Lem} \label{2.0.3}
		Let $(a,b,c)$ be the solution of the degenerate system \eqref{eq:model 1} in dimensions $N=1,2$. Then, there exists positive constants $\hat{K}$ and $K_c>0$ and some $\mu_c\in\mathbb{N}$, independent of time, such that
		\begin{equation*}
			\left \{
			\begin{aligned}
				\Vert b \Vert_{\mathrm{L}^{\frac{3}{2}}(\Omega)}&\leq \hat{K}(1+\tau)^{\frac{5}{6}} \qquad \ \forall  \tau\geq 0,
				\\
				\Vert a \Vert_{\mathrm{L}^{3.5}(\Omega_{\tau,\tau+1})} &\leq K_c(1+\tau)^{\mu_c} \qquad \forall \tau\geq 0,
				\\
				\Vert c \Vert_{\mathrm{L}^{3.5}(\Omega_{\tau,\tau+1})} &\leq K_c(1+\tau)^{\mu_c} \qquad \forall \tau\geq 0.
			\end{aligned}
			\right .
		\end{equation*}
	\end{Lem}
	\begin{proof}
		The proof is similar to the proof of proposition \ref{N=3 main}. We use Gagliardo-Nirenberg inequality which says there exists a positive constant depending on the domain and dimension $C_2=C_2(\Omega,N)$ such that 
		$$ \Vert D^j \sqrt{a} \Vert_{\mathrm{L}^p(\Omega)}\leq C_2 \Vert D^m \sqrt{a} \Vert_{\mathrm{L}^r(\Omega)}^{\alpha} \Vert \sqrt{a} \Vert_{\mathrm{L}^q(\Omega)}^{1-\alpha}+\Tilde{C_2}(\Omega,N,r,q)\Vert \sqrt{a} \Vert_{L^q(\Omega)},$$ 
		where 
		\[
		 1 \leq p,q,r \leq +\infty \ \mbox{and} \ 
		\frac{1}{p}=\frac{j}{N}+\left(\frac{1}{r}-\frac{m}{N}\right)\alpha\ + \frac{1-\alpha}{q} \ \ \mbox{with} \ \ \frac{j}{N}\leq \alpha \leq 1.
		\]
		For dimension $N=2$, we choose $r=2,m=1,j=0,q=2$ and  $\alpha=\frac{1}{2}$. This particular choice yields the following
		\[ 
		\Vert  \sqrt{a} \Vert_{\mathrm{L}^4(\Omega)}^4\leq 2^3 C_2^4 \Vert \nabla \sqrt{a} \Vert_{\mathrm{L}^2(\Omega)}^{2} \Vert \sqrt{a} \Vert_{\mathrm{L}^2(\Omega)}^2+2^3 \Tilde{C_2}^4\Vert \sqrt{a} \Vert_{\mathrm{L}^2(\Omega)}^4.
		\]
		Just like in proposition \ref{N=3 main}, the following estimates hold
		\[
		\int_{\tau}^{\tau+2} \Vert \sqrt{a} \Vert_{W^{1,2}(\Omega)}^2  \leq K_1,\ \
		\int_{\tau}^{\tau+2} \Vert \sqrt{c} \Vert_{W^{1,2}(\Omega)}^2  \leq K_1, \quad \forall \tau\geq 0,
		\]
		where  $K_1=\frac{2E(a_0,b_0,c_0)}{d_a}+\frac{2E(a_0,b_0,c_0)}{d_c}+2M_1\vert\Omega\vert +2M_2\vert\Omega\vert$.
		We consider the relation we obtained from Gagliardo-Nirenberg inequality. Integrating that relation from time  $\tau$ to $\tau+2$ and thanks to our previous estimates, we conclude the following
		\[ 
		\int_{\tau}^{\tau+2}\Vert  \sqrt{a} \Vert_{\mathrm{L}^4(\Omega)}^4 \leq 2^4 C_2^4 M_1^2 \vert \Omega \vert^2 K_1+2^4 \Tilde{C_2}^4 M_1^4,
		\]
		\[
		 \int_{\tau}^{\tau+2}\Vert  \sqrt{c} \Vert_{\mathrm{L}^4(\Omega)}^4 \leq 2^4C_2^4 M_1^2 \vert \Omega \vert^2 K_1+2^4\Tilde{C_2}^4 M_1^4.
		 \]
		Let's choose $K_2= (2^4 C_2^4 M_1^2 \vert \Omega \vert^2 K_1+2^4 \Tilde{C_2}^4 M_1^4)^{\frac{1}{2}}$. we arrive at the following estimates  $\forall \tau \geq 0$.
		\[ 
		\Vert a \Vert_{\mathrm{L}^2(\Omega_{\tau,\tau+2})} \leq K_2, \qquad \Vert c \Vert_{\mathrm{L}^2(\Omega_{\tau,\tau+2})} \leq K_2.
		 \]
		Applying H\"older inequality we arrive at the following estimates 
		\begin{gather}
			\Vert a \Vert_{\mathrm{L}^{\frac{5}{3}}(\Omega_{\tau,\tau+2})} \leq 2K_2\vert \Omega \vert^{\frac{1}{6}}, \qquad \Vert c \Vert_{\mathrm{L}^{\frac{5}{3}}(\Omega_{\tau,\tau+2})} \leq 2K_2 \vert \Omega \vert^{\frac{1}{6}} \quad \forall \tau\geq 0. 
		\end{gather}
		Rest of the proof is similar to the proof of proposition \ref{N=3 main}. We obtain a time independent constant $\hat{K} > 0$  such that for all $t \geq 0$
		$$
		\Vert b \Vert_{\mathrm{L}^{\frac{3}{2}}(\Omega)} \leq \hat{K} (1+t)^{\frac{5}{6}}.
		$$
		For dimension $N=1$, choose $r=2,m=1,j=0,q=2$ and  $\alpha=\frac{1}{3}$. Arguing similarly like in the case of dimension $N=2$, we obtain  time independent constants $\hat{K} >0,K_c>0$ and $\mu_c \in \mathbb{N}$  such that for all  $\tau \geq 0$
		\begin{align*}
			\Vert b \Vert_{\mathrm{L}^{\frac{3}{2}}(\Omega)}&\leq \hat{K}(1+\tau)^{\frac{5}{6}} \qquad \  \tau\geq 0,
			\\
			\Vert a \Vert_{\mathrm{L}^{3.5}(\Omega_{\tau,\tau+1})} &\leq K_c(1+\tau)^{\mu_c} \qquad \tau\geq 0,
			\\
			\Vert c \Vert_{\mathrm{L}^{3.5}(\Omega_{\tau,\tau+1})} &\leq K_c(1+\tau)^{\mu_c} \qquad \tau\geq 0.
		\end{align*}
\end{proof}

	Above Proposition \ref{L^p with time} and Proposition \ref{N=3 main} will help us to relate entropy dissipation with the missing term  $\delta_{B}$. Furthermore the integral estimates of $c$ will provide us polynomial growth of the solutions too. We begin by the polynomial growth of the solutions in the next theorem.

	\begin{Lem}\label{L^infty}
		Let $(a,b,c)$ be the solution of degenerate system \eqref{eq:model 1}. Furthermore, let for $p> \frac{N+2}{2}$, there exists a time independent constant $K_c>0$ and $\mu_c\in\mathbb{N}$, such that
		\[
		\Vert a \Vert_{\mathrm{L}^{p}(\Omega_{\tau,\tau+1})}+\Vert c \Vert_{\mathrm{L}^{p}(\Omega_{\tau,\tau+1})}\leq K_c(1+\tau)^{\mu_c} \quad \forall \tau\geq0
		\]
		Then there exists constant $K_{\infty}$ and $\mu \in \mathbb{N}$, such that
		\begin{equation*}
			\left \{
			\begin{aligned}
				\Vert a \Vert_{\mathrm{L}^{\infty}(\Omega_{t})} \leq & K_{\infty}(1+t)^{\mu} \qquad \forall t\geq 0\\
				\Vert b \Vert_{\mathrm{L}^{\infty}(\Omega_{t})} \leq & K_{\infty}(1+t)^{\mu} \qquad \forall t\geq 0\\
				\Vert c \Vert_{\mathrm{L}^{\infty}(\Omega_{t})} \leq & K_{\infty}(1+t)^{\mu} \qquad \forall t\geq 0.
			\end{aligned}
			\right.
		\end{equation*}
	\end{Lem}
	
	\vspace{.2cm}
	
	\begin{proof}
		The species $a$ is a positive subsolution of the following equation in the time interval $(0,2)$.
		\begin{equation}\label{eq:tilde a}
			\left \{
			\begin{aligned}
				\partial_{t}a(t,x)-d_a\Delta a(t,x)
				\leq &c  \qquad \qquad \qquad  \mbox{ in }\Omega_{0,2}
				\\
				\nabla a(t,x)\cdot n=&  0 \qquad   \qquad \quad   \ \  \mbox{ on } \partial\Omega_{0,2} 
				\\
				a(0,x) =& a_0 \qquad   \quad  \qquad \ \ \mbox{ in } \Omega.
			\end{aligned}
			\right.
		\end{equation}
		The solution corresponding to the equation \eqref{eq:tilde a} can be expressed as:
		\[
		a\leq \int_{\Omega}G_{d_a}(t,0,x,y)a_0(y) \rm{d}y+ \int_{0}^{t}\int_{\Omega} G_{d_a}(t,s,x,y)c(s,y)\rm{d}y\rm{d}s.
		\]
		where $G_{d_a}$ denotes the Green's function associated with the operator $\partial_t-d_a\Delta$ with Neumann boundary condition. We use the fact that $\displaystyle{a_0(y)\in\mathrm{L}^{\infty}(\Omega)}$ and $\displaystyle{\int_{\Omega}G_{d_a}(t,0,x,y) \rm{d}y \leq 1}$ for all $t\in(0,2)$. It yields
		\[
		\left\vert \int_{\Omega}G_{d_a}(t,0,x,y)a_0(y) \rm{d}y \right\vert \leq \Vert a_0\Vert_{\mathrm{L}^{\infty}(\Omega)}.
		\]
		We use the following Green's function estimate is from \cite{Morra83}\cite{ML15}: there exists a constant $K_1>0$, depending only on the domain, such that
		\[
		0\leq G_{d_a}(t_1,s,x,y) \leq \frac{K_1}{(4\pi( t_1-s))^\frac{N}{2}}e^{-\kappa \frac{\Vert x-y \Vert^2}{(t_1-s)}} =: g_{d_a}(t_1-s,x-y),
		\]
		for some constant $\kappa>0$ depending only on $\Omega$ and the diffusion coefficient $d$. Note that
		\[
		\left\Vert g_d\right\Vert_{\mathrm{L}^{z}(\Omega_{0,2})} \leq K_2 \qquad \forall z\in\left[1,1+\frac{N}{2}\right)
		\]
		for some constant positive $K_2$ depending on $z$. We in particular choose $\displaystyle{\frac{1}{z}=1-\frac{1}{p}}$ which is admissible because $\displaystyle{p>\frac{N+2}{2}}$. Hence 
		\[
		\left \vert \int_{0}^{t}\int_{\Omega} G_{d_a}(t,s,x,y)c(s,y)\rm{d}y\rm{d}s\right\vert \leq 23^{\mu_c}K_2K_c.
		\]
		It further yields the pointwise bound for the species $a$ in the unit parabolic cylinder in initial time. More precisely, we have
		\begin{align}\label{new label: Linfty a}
		\Vert a\Vert_{\mathrm{L}^{\infty}(\Omega_{0,2})}\leq  \Vert a_0\Vert_{\mathrm{L}^{\infty}(\Omega)} +23^{\mu_c}K_2K_c.
		\end{align}
Next we consider the differential relation $\partial_t b\leq c$. It yields
\begin{align}\label{new label: b-c}
b\leq b_0+\int_{0}^{t}c.
\end{align}
Employing Minkowski's integral inequality on the above estimate we obtain
\[
\Vert b \Vert_{\mathrm{L}^{p}(\Omega_{0,2})}\leq 2^{\frac{1}{p}} \left(\Vert b_0 \Vert_{\mathrm{L}^{\infty}(\Omega)}\vert \Omega\vert^{\frac{1}{p}}+ 3^{\mu_c}2^{\frac{2p-1}{p}}K_c\right).
\]
Using the estimate \eqref{new label: Linfty a}, we derive the following estimate
\[
\Vert ab \Vert_{\mathrm{L}^{p}(\Omega_{0,2})}\leq 2^{\frac{1}{p}}\left(\Vert a_0\Vert_{\mathrm{L}^{\infty}(\Omega)} +2K_2K_c\right) \left(\Vert b_0 \Vert_{\mathrm{L}^{\infty}(\Omega)}\vert \Omega\vert^{\frac{1}{p}}+ 3^{\mu_c}2^{\frac{2p-1}{p}}K_c\right)=:K_{b,[0,1]}.
\]
Similar to the arguments for the species $a$, the above estimates yields the following pointwise estimate for the species $c$
\[
\Vert c\Vert_{\mathrm{L}^{\infty}(\Omega_{0,2})}\leq  \Vert c_0\Vert_{\mathrm{L}^{\infty}(\Omega)} +3^{\mu_c}K_2K_{b,[0,1]}=:K_{c,[0,1]}.
\]
The pointwise estimate of the species $b$ in unit time is a consequence of the above estimate. The relation \eqref{new label: b-c} and the non-negativity of the species $b$, yields
\[
\Vert b\Vert_{\mathrm{L}^{\infty}(\Omega_{0,2})}\leq \Vert b_0\Vert_{\mathrm{L}^{\infty}(\Omega)}+2K_{c,[0,1]}.
\]
Consider the following constant $\displaystyle{K_{\infty,[0,2]}:=\max \left \{ \Vert a_0\Vert_{\mathrm{L}^{\infty}(\Omega)} +23^{\mu_c}K_2K_c, K_{c,[0,1]}, \Vert b_0\Vert_{\mathrm{L}^{\infty}(\Omega)}+2K_{c,[0,1]}\right\}}$. This helps us deduce
\[
\Vert a\Vert_{\mathrm{L}^{\infty}(\Omega_{0,2})}, \Vert b\Vert_{\mathrm{L}^{\infty}(\Omega_{0,2})}, \Vert c\Vert_{\mathrm{L}^{\infty}(\Omega_{0,2})} \leq K_{\infty,[0,2]}.
\]

		Let $\phi:[0,\infty)\rightarrow [0,1]$ be a smooth function such that $\phi(0)=0$, $\phi\big\vert_{[1,\infty)}=1 \text{\ and \ } \phi'\in[0,M_{\phi}]$ for some positive constant $M_{\phi}$. Let's denote $\phi_{\tau}(s)=\phi(s-\tau)$. The product function $\phi_{\tau}a$ satisfies the following initial-boundary value problem:
		
		\begin{equation} \nonumber
			\left\{
			\begin{aligned}
				\partial_t \phi_{\tau}a-d_a \Delta \phi_{\tau}a = & \phi'_{\tau}a+\phi_{\tau}( c-ab) \qquad \mbox{ in } \Omega_{\tau,\tau+2}\\
				\nabla \phi_{\tau}a .n=& 0 \qquad \qquad  \  \ \qquad \qquad\ \mbox { on } \partial\Omega_{\tau,\tau+2}\\
				\phi_{\tau}a(\tau,x)=& 0 \qquad \qquad \qquad \  \qquad  \ \  \mbox{ in }\Omega.
			\end{aligned}
			\right.
		\end{equation}
		
		Applying integrability estimation (see Theorem \ref{estimation 1} from the Appendix) on the above equation, we obtain
		\begin{align}
			\Vert \phi_{\tau}a \Vert_{\mathrm{L}^{\infty}(\Omega_{\tau,\tau+2})} \leq& C_{IE}(\Omega,N,p)\Vert \phi'_{\tau}a+\phi_{\tau} c \Vert_{\mathrm{L}^{p}(\Omega_{\tau,\tau+1})}\nonumber\\
			\Vert a \Vert_{\mathrm{L}^{\infty}(\Omega_{\tau+1,\tau+2})}\leq & C_{IE}(\Omega,N,p)(1+M_{\phi})K_c(1+\tau)^{\mu_c}. \label{kappu 1}
		\end{align}
		
		Next we consider the differential relation $\partial_t b\leq c$. It yields
		\[
			b\leq b_0+\int_{0}^{t}c.
		\]
		Employing Minkowski's integral inequality on the above estimate we obtain
		\[
		\Vert b \Vert_{\mathrm{L}^{p}(\Omega_{0,\tau})}\leq (1+\tau)^{\frac{1}{p}} \left(\Vert b_0 \Vert_{\mathrm{L}^{\infty}(\Omega)}\vert \Omega\vert^{\frac{1}{p}}+ (1+\tau)^{1+\mu_c}K_c\right).
		\]
		We derive the following integral estimate of $ab$
		\[
		\Vert ab \Vert_{\mathrm{L}^{p}(\Omega_{\tau,\tau+1})} \leq C_{IE}(\Omega,N,p)(1+M_{\phi})K_c(1+\tau)^{\mu_c}+(1+\tau)^{\frac{1}{p}} \left(\Vert b_0 \Vert_{\mathrm{L}^{\infty}(\Omega)}\vert \Omega\vert^{\frac{1}{p}}+ (1+\tau)^{1+\mu_c}K_c\right)
		\]
		where $p>\frac{N+2}{2}$. Let us denote the following positive constant 
		$
		\displaystyle{K_3:=C_{IE}(\Omega,N,p)(1+M_{\phi})K_c+\left(\Vert b_0 \Vert_{\mathrm{L}^{\infty}(\Omega)}\vert \Omega\vert^{\frac{1}{p}}+ K_c\right)}
		$.
		 It yields
		\[
		\Vert ab \Vert_{\mathrm{L}^{p}(\Omega_{\tau,\tau+1})}\leq K_3(1+t)^{2+\mu_c} \qquad \mbox{where} \ p>\frac{N+2}{2}.
		\]
		Considering the same smooth cutoff function $\phi$
		
		\begin{equation} \nonumber
			\left\{
			\begin{aligned}
				\partial_t \phi_{\tau}c-d_c \Delta \phi_{\tau}c = & \phi'_{\tau}c+\phi_{\tau}( ab-c) \qquad \mbox{ in } \Omega_{\tau,\tau+2}\\
				\nabla \phi_{\tau}c. n=& 0 \qquad \qquad  \  \ \qquad \qquad\ \mbox { on } \partial\Omega_{\tau,\tau+2}\\
				\phi_{\tau}c(\tau,x)=& 0 \qquad \qquad \qquad \  \qquad  \ \  \mbox{ in }\Omega
			\end{aligned}
			\right.
		\end{equation}
		
		Integrability estimation(see Theorem \ref{estimation 1} from Appendix) on the above equation yields
		\begin{align}
			\Vert \phi_{\tau}c \Vert_{\mathrm{L}^{\infty}(\Omega_{\tau,\tau+2})} \leq& C_{IE}(\Omega,N,p)\Vert \phi'_{\tau}c+\phi_{\tau} ab \Vert_{\mathrm{L}^{p}(\Omega_{\tau,\tau+1})}\nonumber\\
			\Vert c \Vert_{\mathrm{L}^{\infty}(\Omega_{\tau+1,\tau+2})}\leq & C_{IE}(\Omega,N,p)(1+M_{\phi})(K_c+K_3) (1+\tau)^{2+\mu_c}. \label{kappu 2}
		\end{align}
		
	The pointwise estimate of the species $b$ in unit time is a consequence of the above estimate. The relation \eqref{new label: b-c} and the non-negativity of the species $b$, yields
		\begin{align*}
		\Vert b \Vert_{\mathrm{L}^{\infty}(\Omega_{\tau})} \leq & (1+\tau)\Vert b_0 \Vert_{\mathrm{L}^{\infty}(\Omega)}+\Vert c \Vert_{\mathrm{L}^{\infty}(\Omega_{\tau,\tau+1})}(1+\tau)\\
		\leq & (1+\tau)\Vert b_0 \Vert_{\mathrm{L}^{\infty}(\Omega)}+ \bigg(C_{IE}(\Omega,N,p)+K_{\infty,[0,2]}\bigg)(1+M_{\phi})(K_c+K_3) (1+\tau)^{3+\mu_c}.
		\end{align*}
		
		By choosing the constant 
		\begin{equation*}
		K_{\infty}:=\max
		\left\{
		\begin{aligned}
		&K_{\infty,[0,2]}, C_{IE}(\Omega,N,p)(1+M_{\phi})K_c, C_{IE}(\Omega,N,p)(1+M_{\phi})(K_c+K_3),\\
		&\Vert b_0 \Vert_{\mathrm{L}^{\infty}(\Omega)}+ \bigg(C_{IE}(\Omega,N,p)+K_{\infty,[0,2]}\bigg)(1+M_{\phi})(K_c+K_3) 
		\end{aligned}
	\right \}
\end{equation*}
and $\mu=3+\mu_c$, we arrive at our result.
		\end{proof}
	
\vspace{.2cm}		
		
		Note for the dimension $N\geq 4$, if $d_a,d_c$ satisfying closeness condition \eqref{closeness condition imp} for $p>N\geq \frac{N+2}{2}$, the assumption in the Theorem \ref{L^infty} is automatically satisfied (Proposition-\ref{L^p with time}), however the assumption also holds true for $N=1,2,3$ (Proposition \ref{N=3 main} and Lemma \ref{2.0.3}) regardless of any smallness condition on the non-zero diffusion coefficients.

	\vspace{.3cm}

	Next we show the integral estimate of $b$, as in Lemma \ref{N/2 estimate} and Proposition-\ref{N=3 main}, helps us to relate the missing term  $\delta_{B}$  with the entropy dissipation functional. The relation is described in the following proposition:

\begin{Prop}\label{main relatio}
Let $N\geq 4$ and let $(a,b,c)$ be the solution to the degenrate system \eqref{eq:model 1}. Let the nonzero diffusion coefficients $d_a,d_c$ satisfy the closeness condition \eqref{closeness condition imp}. Then the entropy dissipation $D(a,b,c)$ satisfies
\begin{align}\label{eq:main-relatio}
D(a,b,c) \ge \mathcal{K}\, (1+t)^{-\frac{N-2}{N-1}} \left( \left\Vert \delta_A \right\Vert^2_{\mathrm L^2(\Omega)} + \left\Vert \delta_B \right\Vert^2_{\mathrm L^2(\Omega)} + \left\Vert \delta_C \right\Vert^2_{\mathrm L^2(\Omega)} \right)
\qquad
\mbox{ for }t\ge0,
\end{align}
where the positive constant $\mathcal{K}$ depends only on the dimension N, the domain $\Omega$, the constants $M_1$ and $M_2$ in the mass conservation properties \eqref{mass conserve 1}-\eqref{mass conserve 2} and the nonzero diffusion coefficients $d_a$ and $d_c$.
\end{Prop}
\begin{proof}
We rewrite the entropy dissipation functional as
\begin{align}\label{eq:dissip-rewrite}
D(a,b,c) = 4d_a \int_{\Omega}\left\vert \nabla A \right\vert^2\, {\rm d}x + 4d_c \int_{\Omega}\left\vert \nabla C \right\vert^2\, {\rm d}x + \int_{\Omega} (ab-c)\ln\left(\frac{ab}{c}\right)\, {\rm d}x.
\end{align}
We recall an algebraic identity which says that for all $p,q\ge0$, there holds $(p-q)\left(\ln p-\ln q\right) \geq 4\left(\sqrt{p}-\sqrt{q}\right)^2$. Using this algebraic identity in the last term and employing the Poincar\'e-Wirtinger inequality (see Theorem \ref{Poincare-Wirtinger} from the Appendix for the precise statement) for the first two terms of the above dissipation functional yields
\begin{align}\label{eq:dissip-bound-1}
D(a,b,c) \geq \frac{4d_a}{P(\Omega)} \left\Vert \delta_A \right\Vert^2_{\mathrm L^{\frac{2N}{N-2}}(\Omega)} + \frac{4d_c}{P(\Omega)} \left\Vert \delta_C \right\Vert^2_{\mathrm L^{\frac{2N}{N-2}}(\Omega)} + 4 \left\Vert AB - C \right\Vert^2_{\mathrm L^2(\Omega)},
\end{align}
where $P(\Omega)=C\left(\Omega,\frac{2N}{N-2}\right)$ is the Poincar\'e constant. From the above inequality, it follows that
\begin{align}\label{eq:dissip-bound-2}
D(a,b,c) \geq \frac{4d_a}{P(\Omega)} \left\Vert \delta_A \right\Vert^2_{\mathrm L^{\frac{2N}{N-2}}(\Omega)} + \frac{4d_c}{P(\Omega)} \left\Vert \delta_C \right\Vert^2_{\mathrm L^{\frac{2N}{N-2}}(\Omega)} + \eta \left\Vert AB - C \right\Vert^2_{\mathrm L^2(\Omega)},
\end{align}
for any $0\le \eta\le 4$. As a consequence of the H\"older inequality, the inequality \eqref{eq:dissip-bound-1} leads to
\[
D(a,b,c) \geq \frac{4d_a\left\vert\Omega\right\vert^{-\frac{2}{N}}}{P(\Omega)} \left\Vert \delta_A \right\Vert^2_{\mathrm L^2(\Omega)} + \frac{4d_c\left\vert\Omega\right\vert^{-\frac{2}{N}}}{P(\Omega)} \left\Vert \delta_C \right\Vert^2_{\mathrm L^2(\Omega)} + 4 \left\Vert AB - C \right\Vert^2_{\mathrm L^2(\Omega)}.
\]
It is apparent from the above lower bound that a term involving $\left\Vert \delta_B\right\Vert^2_{\mathrm L^2(\Omega)}$ is missing. To arrive at a lower bound involving this missing term, we fix an arbitrary constant $\varepsilon>0$ (to be chosen later) and distinguish two cases: a case corresponding to 
\begin{align}\label{eq:case-bounded above}
\max\left\{ \left\Vert \delta_A \right\Vert^2_{\mathrm L^2(\Omega)} , \left\Vert \delta_C \right\Vert^2_{\mathrm L^2(\Omega)} \right\} \le \varepsilon
\end{align}
and another case corresponding to
\begin{align}\label{eq:case- NOT bounded above}
\max\left\{ \left\Vert \delta_A \right\Vert^2_{\mathrm L^2(\Omega)} , \left\Vert \delta_C \right\Vert^2_{\mathrm L^2(\Omega)} \right\} > \varepsilon.
\end{align}
We first treat the case corresponding to \eqref{eq:case-bounded above}. Observe that
\[
\left\Vert AB - C \right\Vert^2_{\mathrm L^2(\Omega)} 
= \left\Vert \left(\delta_A + \overline{A}\right)B - \left(\delta_C + \overline{C}\right) \right\Vert^2_{\mathrm L^2(\Omega)}
= \left\Vert \left( \, \overline{A}B - \overline{C}\right) + \left( B \delta_A -\delta_C\right) \right\Vert^2_{\mathrm L^2(\Omega)}
\]
We recall an algebraic identity which says that for all $p,q\in\mathbb{R}$, there holds $(p-q)^2\ge \frac{p^2}{2} - q^2$. Using this algebraic identity in the above equality, we obtain
\begin{align*}
\left\Vert AB - C \right\Vert^2_{\mathrm L^2(\Omega)} & \ge \frac12 \left\Vert \overline{A}B - \overline{C} \right\Vert^2_{\mathrm L^2(\Omega)} - \left\Vert B \delta_A -\delta_C \right\Vert^2_{\mathrm L^2(\Omega)}
\\
& \ge \frac12 \left\Vert \overline{A}B - \overline{C} \right\Vert^2_{\mathrm L^2(\Omega)} - 2 \left\Vert B \delta_A \right\Vert^2_{\mathrm L^2(\Omega)} - 2 \left\Vert \delta_C \right\Vert^2_{\mathrm L^2(\Omega)},
\end{align*}
thanks to the algebraic identity $(p-q)^2\le 2p^2+2q^2$ for $p,q\in\mathbb{R}$. Employing the H\"older inequality in the second and third terms of the lower bound in the above inequality results in
\[
\left\Vert AB - C \right\Vert^2_{\mathrm L^2(\Omega)} \ge \frac12 \left\Vert \overline{A}B - \overline{C} \right\Vert^2_{\mathrm L^2(\Omega)} - 2 \left\Vert b \right\Vert_{\mathrm L^\frac{N}{2}(\Omega)} \left\Vert \delta_A \right\Vert^2_{\mathrm L^\frac{2N}{N-2}(\Omega)} - 2 \left\vert\Omega\right\vert^\frac{2}{N} \left\Vert \delta_C \right\Vert^2_{\mathrm L^\frac{2N}{N-2}(\Omega)}.
\]
Next, using the key integrability estimate on $\left\Vert b(t,\cdot)\right\Vert_{\mathrm L^\frac{N}{2}(\Omega)}$ from Lemma \ref{N/2 estimate} yields
\begin{align}\label{eq:dissip-bound-3}
\left\Vert AB - C \right\Vert^2_{\mathrm L^2(\Omega)} \ge \frac12 \left\Vert \overline{A}B - \overline{C} \right\Vert^2_{\mathrm L^2(\Omega)} - 2 K_3\, (1+t)^{\frac{N-2}{N-1}} \left\Vert \delta_A \right\Vert^2_{\mathrm L^\frac{2N}{N-2}(\Omega)} - 2 \left\vert\Omega\right\vert^\frac{2}{N} \left\Vert \delta_C \right\Vert^2_{\mathrm L^\frac{2N}{N-2}(\Omega)}.
\end{align}
In order to relate the above lower bound to $\left\Vert\delta_B\right\Vert_{\mathrm L^2(\Omega)}$, we further analyse the following term:
\[
\left\Vert \overline{A}B - \overline{C} \right\Vert^2_{\mathrm L^2(\Omega)}.
\]
Note that if $\overline{A}\ge \sqrt{\varepsilon}$, then 
\begin{align}\label{eq:dissip-bound-4}
\left\Vert \overline{A}B - \overline{C} \right\Vert^2_{\mathrm L^2(\Omega)} \ge \varepsilon \left\Vert B - \overline{B} \right\Vert^2_{\mathrm L^2(\Omega)}.
\end{align}
To see this, factorising $\overline{A}B$ as $\overline{C}(1+\mu(x))$, we get
\[
\left\Vert \overline{A}B - \overline{C} \right\Vert^2_{\mathrm L^2(\Omega)} = \overline{C}^2 \overline{\mu^2} \left\vert \Omega\right\vert
\quad
\mbox{ and }
\quad
\left\Vert B - \overline{B} \right\Vert^2_{\mathrm L^2(\Omega)} = \frac{\overline{C}^2}{\overline{A}^2} \left\Vert \mu - \overline{\mu} \right\Vert^2_{\mathrm L^2(\Omega)} 
\le \frac{\overline{C}^2}{\overline{A}^2}\overline{\mu^2}\left\vert\Omega\right\vert
\le \frac{\overline{C}^2\overline{\mu^2}\left\vert\Omega\right\vert}{\varepsilon}.
\]
On the other hand, let us consider the case when $\overline{A} < \sqrt{\varepsilon}$. Note that
\[
\left\Vert \delta_A\right\Vert^2_{\mathrm L^2(\Omega)} = \left\vert\Omega\right\vert \left( \overline{A^2} - \overline{A}^2 \right)
\implies \overline{A^2} \le \varepsilon \left(1 + \frac{1}{\left\vert\Omega\right\vert}\right),
\]
where we have also used the fact that we are dealing with the case $\left\Vert \delta_A \right\Vert^2_{\mathrm L^2(\Omega)}\le \varepsilon$. Observe that
\begin{align*}
\overline{C}^2 = \overline{C^2} - \frac{1}{\left\vert\Omega\right\vert} \left\Vert \delta_C\right\Vert^2_{\mathrm L^2(\Omega)}
& = \overline{C^2} + \overline{A^2} - \overline{A^2} - \frac{1}{\left\vert\Omega\right\vert} \left\Vert \delta_C\right\Vert^2_{\mathrm L^2(\Omega)}
\\
& \ge M_1 - \varepsilon \left( 1 + \frac{2}{\left\vert\Omega\right\vert}\right),
\end{align*}
thanks to the mass conservation property \eqref{mass conserve 1}, the bound on $\overline{A^2}$ from above and the fact that $\left\Vert \delta_C \right\Vert^2_{\mathrm L^2(\Omega)}\le \varepsilon$. Now, using the algebraic identity $(p-q)^2\ge \frac{p^2}{2} - q^2$, we arrive at
\[
\left\Vert \overline{A}B - \overline{C} \right\Vert^2_{\mathrm L^2(\Omega)} = \frac{\left\vert\Omega\right\vert}{2} \left( \overline{C}^2 - 2 \overline{A}^2 \overline{B^2}\right)
\ge \frac{\left\vert\Omega\right\vert}{2} \left( M_1 - \varepsilon \left( 1 + \frac{2}{\left\vert\Omega\right\vert}\right) - 2 \varepsilon M_2 \right)
\]
where we have used the aforementioned lower bound for $\overline{C}^2$, the mass conservation property \eqref{mass conserve 2} and that $\overline{A} < \sqrt{\varepsilon}$. Let us now choose
\begin{align}\label{eq:epsilon-choice}
\varepsilon := \frac{M_1}{2} \frac{\left\vert\Omega\right\vert}{\left\vert\Omega\right\vert + 2 + 2 M_2 \left\vert\Omega\right\vert}.
\end{align}
With the above choice of $\varepsilon$, we obtain
\begin{align}\label{eq:dissip-bound-5}
\left\Vert \overline{A}B - \overline{C} \right\Vert^2_{\mathrm L^2(\Omega)} \ge \frac{M_1\left\vert\Omega\right\vert}{4} \ge \frac{M_1}{4M_2} \left\Vert B - \overline{B}\right\Vert^2_{\mathrm L^2(\Omega)},
\end{align}
where we have used the observation \eqref{eq:delta-B-bound-1} from earlier. Using \eqref{eq:dissip-bound-4} and \eqref{eq:dissip-bound-5} in \eqref{eq:dissip-bound-3} and \eqref{eq:dissip-bound-2} helps us deduce that
\begin{align*}
D(a,b,c) 
& \geq \left( \frac{4d_a}{P(\Omega)} - 2\, \eta\, K_3\, (1+t)^{\frac{N-2}{N-1}}\right)\left\Vert \delta_A \right\Vert^2_{\mathrm L^{\frac{2N}{N-2}}(\Omega)} 
+ \left( \frac{4d_c}{P(\Omega)} - 2\, \eta \left\vert\Omega\right\vert^\frac{2}{N}\right) \left\Vert \delta_C \right\Vert^2_{\mathrm L^{\frac{2N}{N-2}}(\Omega)}
+ \eta \varepsilon \left\Vert \delta_B \right\Vert^2_{\mathrm L^2(\Omega)}
\end{align*}
for the case $\overline{A}\ge \sqrt{\varepsilon}$ and that 
\begin{align*}
D(a,b,c) 
& \geq \left( \frac{4d_a}{P(\Omega)} - 2\, \eta\, K_3\, (1+t)^{\frac{N-2}{N-1}}\right)\left\Vert \delta_A \right\Vert^2_{\mathrm L^{\frac{2N}{N-2}}(\Omega)} 
+ \left( \frac{4d_c}{P(\Omega)} - 2\, \eta \left\vert\Omega\right\vert^\frac{2}{N}\right) \left\Vert \delta_C \right\Vert^2_{\mathrm L^{\frac{2N}{N-2}}(\Omega)}
+ \eta \frac{M_1}{8M_2} \left\Vert \delta_B \right\Vert^2_{\mathrm L^2(\Omega)}
\end{align*}
for the case $\overline{A}< \sqrt{\varepsilon}$. Next, observe that by taking
\[
\eta(t) := \left( \frac{2\min\{d_a,d_c,2\}}{P(\Omega) \left(K_3 + \left\vert\Omega\right\vert^\frac{2}{N}\right) + 1} \right) (1+t)^{-\frac{N-2}{N-1}},
\]
we obtain 
\[
D(a,b,c) \ge \min\left\{ \frac{M_1\left\vert\Omega\right\vert}{2\left\vert\Omega\right\vert + 4 + 4 M_2 \left\vert\Omega\right\vert}, \frac{M_1}{8M_2} \right\} \eta(t) \left\Vert \delta_B \right\Vert^2_{\mathrm L^2(\Omega)}.
\]
Observe that the above choice of $\eta$ does satisfy $0\le \eta\le4$. 
Hence we deduce in the case corresponding to \eqref{eq:case-bounded above} that
\begin{align*}
D(a,b,c) & = \frac12 D(a,b,c) + \frac12 D(a,b,c)
\\
& \ge \frac{2d_a\left\vert\Omega\right\vert^{-\frac{2}{N}}}{P(\Omega)} \left\Vert \delta_A \right\Vert^2_{\mathrm L^2(\Omega)} 
+ \frac{2d_c\left\vert\Omega\right\vert^{-\frac{2}{N}}}{P(\Omega)} \left\Vert \delta_C \right\Vert^2_{\mathrm L^2(\Omega)}
+ \frac12 \min\left\{ \frac{M_1\left\vert\Omega\right\vert}{2\left\vert\Omega\right\vert + 4 + 4 M_2 \left\vert\Omega\right\vert}, \frac{M_1}{8M_2} \right\} \eta(t) \left\Vert \delta_B \right\Vert^2_{\mathrm L^2(\Omega)}
\\
& \ge \mathcal{K}_1 (1+t)^{-\frac{N-2}{N-1}} \left( \left\Vert \delta_A \right\Vert^2_{\mathrm L^2(\Omega)} + \left\Vert \delta_B \right\Vert^2_{\mathrm L^2(\Omega)} + \left\Vert \delta_C \right\Vert^2_{\mathrm L^2(\Omega)} \right)
\end{align*}
for all $t\ge0$, where
\[
\mathcal{K}_1 := \min\left\{ \frac{2d_a\left\vert\Omega\right\vert^{-\frac{2}{N}}}{P(\Omega)}, \frac{2d_c\left\vert\Omega\right\vert^{-\frac{2}{N}}}{P(\Omega)}, \min\left\{ \frac{M_1\left\vert\Omega\right\vert}{2\left\vert\Omega\right\vert + 4 - 4 M_2 \left\vert\Omega\right\vert}, \frac{M_1}{8M_2} \right\} \left( \frac{\min\{d_a,d_c,2\}}{P(\Omega) \left(K_3 + \left\vert\Omega\right\vert^\frac{2}{N}\right)+1} \right)\right\}.
\]
The case corresponding to \eqref{eq:case- NOT bounded above} is relatively simpler. Observe that	
\begin{align*}
D(a,b,c) & \ge \frac{4\left\vert\Omega\right\vert^{-\frac{2}{N}}}{P(\Omega)} \min\left\{d_a,d_c\right\} \max\left\{ \left\Vert \delta_A \right\Vert^2_{\mathrm L^2(\Omega)} , \left\Vert \delta_C \right\Vert^2_{\mathrm L^2(\Omega)} \right\}
\\
& > \frac{4\varepsilon\left\vert\Omega\right\vert^{-\frac{2}{N}}}{P(\Omega)}\min\left\{d_a,d_c\right\}
\geq \frac{4\varepsilon\left\vert\Omega\right\vert^{-\frac{2}{N}}}{P(\Omega)M_2\left\vert \Omega\right\vert}\min\left\{d_a,d_c\right\} \left\Vert B - \overline{B} \right\Vert^2_{\mathrm L^2(\Omega)},
\end{align*}
where the final inequality is thanks to the following observation:
\begin{align}\label{eq:delta-B-bound-1}
\left\Vert B - \overline{B} \right\Vert^2_{\mathrm L^2(\Omega)} = \left\Vert B\right\Vert^2_{\mathrm L^2(\Omega)} - \overline{B} \left\vert \Omega\right\vert \le \left\Vert B\right\Vert^2_{\mathrm L^2(\Omega)} = \left\Vert b\right\Vert_{\mathrm L^1(\Omega)} \le M_2 \left\vert \Omega\right\vert,
\end{align}
which is a consequence of the mass conservation property \eqref{mass conserve 2}. Hence we deduce in this case that
\begin{align*}
D(a,b,c) & = \frac12 D(a,b,c) + \frac12 D(a,b,c)
\\
& \ge \frac{2d_a\left\vert\Omega\right\vert^{-\frac{2}{N}}}{P(\Omega)} \left\Vert \delta_A \right\Vert^2_{\mathrm L^2(\Omega)} 
+ \frac{2d_c\left\vert\Omega\right\vert^{-\frac{2}{N}}}{P(\Omega)} \left\Vert \delta_C \right\Vert^2_{\mathrm L^2(\Omega)}
+ \frac{2\varepsilon\left\vert\Omega\right\vert^{-\frac{2}{N}}\min\left\{d_a,d_c\right\}}{P(\Omega)M_2\left\vert \Omega\right\vert} \left\Vert \delta_B \right\Vert^2_{\mathrm L^2(\Omega)}
\\
& \ge \mathcal{K}_2 (1+t)^{-\frac{N-2}{N-1}} \left( \left\Vert \delta_A \right\Vert^2_{\mathrm L^2(\Omega)} + \left\Vert \delta_B \right\Vert^2_{\mathrm L^2(\Omega)} + \left\Vert \delta_C \right\Vert^2_{\mathrm L^2(\Omega)} \right)
\end{align*}
for all $t\ge0$, where
\[
\mathcal{K}_2 := \min\left\{ \frac{2d_a\left\vert\Omega\right\vert^{-\frac{2}{N}}}{P(\Omega)}, \frac{2d_c\left\vert\Omega\right\vert^{-\frac{2}{N}}}{P(\Omega)}, \frac{2\varepsilon\left\vert\Omega\right\vert^{-\frac{2}{N}}\min\left\{d_a,d_c\right\}}{P(\Omega)M_2\left\vert \Omega\right\vert}\right\}.
\]
Taking $\mathcal{K} := \min\{\mathcal{K}_1, \mathcal{K}_2\}$ yields the desired result.
\end{proof}
It should be noted that a result similar to that of Proposition \ref{main relatio} can be found when the dimension $N<4$. As in the proof of the above proposition, employing the Poincar\'e-Wirtinger inequality in the expression \eqref{eq:dissip-rewrite} for the entropy dissipation for the case of $N<4$ yields
\[
D(a,b,c) \geq \frac{4d_a}{P(\Omega)} \left\Vert \delta_A \right\Vert^2_{\mathrm L^6(\Omega)} + \frac{4d_c}{P(\Omega)} \left\Vert \delta_C \right\Vert^2_{\mathrm L^6(\Omega)} + 4 \left\Vert AB - C \right\Vert^2_{\mathrm L^2(\Omega)}.
\]
Arguing exactly as in the proof of Proposition \ref{main relatio} and exploiting the bound
\[
\left\Vert b(t,\cdot)\right\Vert_{\mathrm L^\frac32(\Omega)} \le \hat{K}(1+t)^{\frac{5}{6}} \qquad \mbox{ for } \, t\geq 0,
\]
obtained in Proposition \ref{N=3 main} and Lemma \ref{2.0.3} helps us prove the following result. To avoid the repeat of arguments, we skip its proof. The key point to be noted, however, is that this result is unconditional in the sense that the nonzero diffusion coefficients are not assumed to satisfy the closeness condition \eqref{closeness condition imp}.
\begin{Prop}\label{2.0.2}
Let $N< 4$ and let $(a,b,c)$ be the solution to the degenrate system \eqref{eq:model 1}. The entropy dissipation $D(a,b,c)$ satisfies
\[
D(a,b,c) \ge \mathcal{S}\, (1+t)^{-\frac56} \left( \left\Vert \delta_A \right\Vert^2_{\mathrm L^2(\Omega)} + \left\Vert \delta_B \right\Vert^2_{\mathrm L^2(\Omega)} + \left\Vert \delta_C \right\Vert^2_{\mathrm L^2(\Omega)} \right)
\qquad
\mbox{ for }t\ge0,
\]
where the positive constant $\mathcal{S}$ depends only on the dimension N, the domain $\Omega$, the constants $M_1$ and $M_2$ in the mass conservation properties \eqref{mass conserve 1}-\eqref{mass conserve 2} and the nonzero diffusion coefficients $d_a$ and $d_c$.
\end{Prop}
Next, we derive a sub-exponential decay estimate for the relative entropy.
\begin{Prop}\label{decay 1}
Let $N\geq4$ and let $(a,b,c)$ be the solution to the degenerate system \eqref{eq:model 1}. Let $(a_\infty,b_\infty,c_\infty)$ be the associated equilibrium state given by \eqref{eq:equi-state-1}-\eqref{eq:equi-state-2}. Let the nonzero diffusion coefficients $d_a,d_c$ satisfy the closeness condition \eqref{closeness condition imp}. Then, for any given positive $\varepsilon \ll1$, there exists a time $T_{\varepsilon}$ and two positive constants $\mathcal{S}_1$ and $\mathcal{S}_2$ such that
\begin{align}\label{eq:sub-exp-decay-relative-entropy}
E(a,b,c)- E(a_{\infty},b_{\infty},c_{\infty})
\leq \mathcal{S}_1 \, e^{-\mathcal{S}_2(1+t)^{\frac{1-\epsilon}{N-1}}} \qquad \mbox{ for }\, t\ge T_\varepsilon.
\end{align}
\end{Prop} 
\begin{proof}
The relative entropy reads
\begin{align*}
E(a,b,c)-E(a_{\infty},b_{\infty},c_{\infty})
& = \int_{\Omega}\left(a \ln{a}-a-a_{\infty} \ln{a_{\infty}}+a_{\infty}\right)\, {\rm d}x
+\int_{\Omega}\left(b \ln{b}-b-b_{\infty} \ln{b_{\infty}}+b_{\infty}\right)\, {\rm d}x
\\
& \quad + \int_{\Omega}\left(c \ln{c}-c-c_{\infty} \ln{c_{\infty}}+c_{\infty}\right)\, {\rm d}x.
\end{align*}
Using the relation \eqref{eq:equil-relation}, the above expression for the relative entropy becomes
\begin{equation}\label{Entropy Gamma function relation}
\begin{aligned}
E(a,b,c)-E(a_{\infty},b_{\infty},c_{\infty}) & =
\int_{\Omega} \left(a \ln{\frac{a}{a_{\infty}}}-a+a_{\infty}\right)\, {\rm d}x 
+ \int_{\Omega} \left(b \ln{\frac{b}{b_{\infty}}}-b+b_{\infty}\right)\, {\rm d}x
\\
& \quad + \int_{\Omega} \left(c \ln{\frac{c}{c_{\infty}}}-c+c_{\infty}\right)\, {\rm d}x.
\end{aligned}
\end{equation}
Let us define a function $\Phi:(0,\infty)\times(0,\infty)\to\mathbb{R}$ as follows:
\begin{equation}
\Phi(x,y) :=
\left\{
\begin{aligned}
\frac{x \ln\left(\frac{x}{y}\right)-x+y}{\left(\sqrt{x}-\sqrt{y}\right)^2} & \qquad \mbox{ for }x\not=y,
\\
2 & \qquad \mbox{ for }x=y.
\end{aligned}\right.
\end{equation}
It can be shown (see \cite[Lemma 2.1, p.162]{DF06} for details) that the above defined function satisfies the following bound:
\begin{align}\label{eq:Gamma-bound}
\Phi(x,y)\leq C_{\Phi}\max\left\{1,\ln\left(\frac{x}{y}\right)\right\}
\end{align}
for some positive constant $C_\Phi$. Next we rewrite the relative entropy as
\begin{align*}
E(a,b,c)-E(a_{\infty},b_{\infty},c_{\infty}) & =
\int_{\Omega} \Phi(a,a_{\infty})(A-A_{\infty})^2\, {\rm d}x  + \int_{\Omega} \Phi(b,b_{\infty})(B-B_{\infty})^2\, {\rm d}x + \int_{\Omega} \Phi(c,c_{\infty})(C-C_{\infty})^2\, {\rm d}x.
\end{align*}
Note that for any $p\ge\frac12$ and $q>0$, we have
\begin{equation*}
\ln p - \ln q \le \ln (1+\left\vert p\right\vert) + \left\vert \ln q \right\vert
\le 1 + \ln \left\vert p \right\vert + \left\vert \ln q \right\vert
\end{equation*}
and for any $0<p<\frac12$ and $q>0$, we have
\[
\ln p - \ln q \le \left\vert \ln q\right\vert.
\]
This helps us arrive at
\begin{equation}\label{eq:relative-entropy-bd-1}
\begin{aligned}
E(a,b,c) - & E(a_{\infty},b_{\infty},c_{\infty}) 
\\
&\le C_1 \left(1 + \ln(1+t) \right) \left( \left\Vert A - A_\infty\right\Vert^2_{\mathrm L^2(\Omega)} + \left\Vert B - B_\infty\right\Vert^2_{\mathrm L^2(\Omega)} + \left\Vert C - C_\infty\right\Vert^2_{\mathrm L^2(\Omega)} \right)
\end{aligned}
\end{equation}
for all $t\ge0$, where the positive constant $C_1$ is given by
\[
C_1 := C_\Phi\left(1 + \left\vert \ln a_\infty\right\vert + \left\vert \ln b_\infty\right\vert + \left\vert \ln c_\infty\right\vert + \left\vert \ln K_\infty\right\vert + \mu \right).
\]
Here the constants $K_\infty$ and $\mu$ are the ones appearing in the $\mathrm L^\infty(\Omega_t)$ bounds on the concentrations from Lemma \ref{L^infty}. The factor $\left(1 + \ln(1+t) \right)$ in the above estimate of the relative entropy is due to the fact that Lemma \ref{L^infty} says that at least one of the concentrations has a polynomial (in time) bound on its $\mathrm L^\infty$-norm. It should be noted that having an uniform (in time) bound on all of the concentrations gets rid of this time factor. In \cite{DF06}, the authors prove the following bound
\begin{equation}\label{eq:relative-entropy-bd-2}
\begin{aligned}
\left\Vert A - A_\infty\right\Vert^2_{\mathrm L^2(\Omega)} + \left\Vert B - B_\infty\right\Vert^2_{\mathrm L^2(\Omega)} & + \left\Vert C - C_\infty\right\Vert^2_{\mathrm L^2(\Omega)}
\\
& \le C_2 \left( \left\Vert \delta_A \right\Vert^2_{\mathrm L^2(\Omega)} + \left\Vert \delta_B \right\Vert^2_{\mathrm L^2(\Omega)} + \left\Vert \delta_C \right\Vert^2_{\mathrm L^2(\Omega)} + \left\Vert AB - C \right\Vert^2_{\mathrm L^2(\Omega)}\right)
\end{aligned}
\end{equation}
exploiting only the conservation properties \eqref{mass conserve 1}-\eqref{mass conserve 2}. Furthermore, the constant $C_2$ in the above bound depends only on the equilibrium states $A_\infty, B_\infty, C_\infty$ and the constants $M_1,M_2$ from the conservation properties \eqref{mass conserve 1}-\eqref{mass conserve 2} (see \cite[Lemma 3.2, p.168]{DF06} for precise expression for the constant). We thus arrive at the following bound for the relative entropy using \eqref{eq:relative-entropy-bd-1} and \eqref{eq:relative-entropy-bd-2}:
\begin{equation}\label{eq:relative-entropy-bd-3}
\begin{aligned}
E(a,b,c) - & E(a_{\infty},b_{\infty},c_{\infty}) 
\\
&\le C_3 \left(1 + \ln(1+t) \right) \left( \left\Vert \delta_A \right\Vert^2_{\mathrm L^2(\Omega)} + \left\Vert \delta_B \right\Vert^2_{\mathrm L^2(\Omega)} + \left\Vert \delta_C \right\Vert^2_{\mathrm L^2(\Omega)} + \left\Vert AB - C \right\Vert^2_{\mathrm L^2(\Omega)}\right)
\end{aligned}
\end{equation}
where the positive constant $C_3:=C_1C_2$. Hence, thanks to the lower bound \eqref{eq:main-relatio} obtained in Proposition \ref{main relatio} and the lower bound \eqref{eq:dissip-bound-1}, it follows from \eqref{eq:relative-entropy-bd-3} that
\begin{align*}
E(a,b,c) - E(a_{\infty},b_{\infty},c_{\infty}) 
\le C_3 \left(1 + \ln(1+t) \right) \max\left\{\frac{1}{\mathcal{K}}, \frac14 \right\} (1+t)^{\frac{N-2}{N-1}} D(a,b,c).
\end{align*}
Note that for any given positive $\varepsilon \ll1$, there exists a time $T_{\varepsilon}$ such that
\[
\ln(1+t) < (1+t)^\frac{\varepsilon}{N-1} \qquad \mbox{ for all }t\ge T_\varepsilon.
\]
Hence we have
\[
E(a,b,c) - E(a_{\infty},b_{\infty},c_{\infty}) \le C_4 (1+t)^\frac{N-2+\varepsilon}{N-1} D(a,b,c) \qquad \mbox{ for all }t\ge T_\varepsilon,
\]
where the constant $C_4:=C_3\max\left\{\frac{1}{\mathcal{K}}, \frac14 \right\}$. Recall that we have
\[
\frac{{\rm d}}{{\rm d}t} \left( E(a,b,c) - E(a_{\infty},b_{\infty},c_{\infty}) \right) = - D(a,b,c) \qquad \mbox{ for all }t>0.
\]
Thus we have
\[
\frac{{\rm d}}{{\rm d}t} \left( E(a,b,c) - E(a_{\infty},b_{\infty},c_{\infty}) \right) \le - \frac{1}{C_4} (1+t)^{-\frac{N-2+\varepsilon}{N-1}}\left( E(a,b,c) - E(a_{\infty},b_{\infty},c_{\infty}) \right) \qquad \mbox{ for all }t\ge T_\varepsilon.
\]
Integrating the above differential inequality, we obtain 
\begin{align*}
E&(a(t,\cdot),b(t,\cdot),c(t,\cdot)) - E(a_{\infty},b_{\infty},c_{\infty}) 
\\
& \le \left( E(a(T_\varepsilon,\cdot),b(T_\varepsilon,\cdot),c(T_\varepsilon,\cdot)) - E(a_{\infty},b_{\infty},c_{\infty}) \right) e^{-\frac{N-1}{C_4(1-\varepsilon)}(1+T_\varepsilon)^{\frac{1-\varepsilon}{N-1}}} e^{-\frac{N-1}{C_4(1-\varepsilon)}(1+t)^{\frac{1-\varepsilon}{N-1}}}
\\
& \le \left( E(a_0,b_0,c_0) - E(a_{\infty},b_{\infty},c_{\infty}) \right) e^{-\frac{N-1}{C_4(1-\varepsilon)}} e^{-\frac{N-1}{C_4(1-\varepsilon)}(1+t)^{\frac{1-\varepsilon}{N-1}}}.
\end{align*}
We have thus proved the sub-exponential decay \eqref{eq:sub-exp-decay-relative-entropy} of relative entropy with the following explicit constants:
\begin{align*}
\mathcal{S}_1 & = \left( E(a_0,b_0,c_0) - E(a_{\infty},b_{\infty},c_{\infty}) \right) e^{-\frac{N-1}{C_4(1-\varepsilon)}}
\\
\mathcal{S}_2 & = \frac{N-1}{C_4(1-\varepsilon)}.
\end{align*}
\end{proof}
A result similar to that of Proposition \ref{decay 1} can be found when the dimension $N<4$. The proof goes along similar lines and we skip it in the interest of space. The proof banks on the lower bound for the dissipation functional obtained in Proposition \ref{2.0.2}. The key point to be noted, however, is that this result is unconditional in the sense that the nonzero diffusion coefficients are not assumed to satisfy the closeness condition \eqref{closeness condition imp}.
\begin{Prop}\label{decay N=3}
Let $N<4$ and let $(a,b,c)$ be the solution to the degenerate system \eqref{eq:model 1}. Let $(a_\infty,b_\infty,c_\infty)$ be the associated equilibrium state given by \eqref{eq:equi-state-1}-\eqref{eq:equi-state-2}. Then, for any given positive $\varepsilon \ll1$, there exists a time $T_{\varepsilon}$ and two positive constants $\mathcal{S}_3$ and $\mathcal{S}_4$ such that
\[
E(a,b,c)- E(a_{\infty},b_{\infty},c_{\infty})
\leq \mathcal{S}_3 \, e^{-\mathcal{S}_4(1+t)^{\frac{1-\epsilon}{6}}} \qquad \mbox{ for }\, t\ge T_\varepsilon.
\]
\end{Prop}
We are now equipped to prove our main result of this section.
\begin{proof}[Proof of Theorem \ref{convergence theorem 1}:]
We have already obtained sub-exponential decay (in time) of the relative entropy in Proposition \ref{decay 1} (for dimension $N\ge4$) and in Proposition \ref{decay N=3} (for dimension $N<4$). Hence the sub-exponential decay in the $\mathrm L^1$-norm is a direct consequence of the following Czisz\'ar-Kullback-Pinsker type inequality that relates relative entropy and the $\mathrm L^1$-norm:
\begin{align*}
E(a,b,c) - E(a_{\infty},b_{\infty},c_{\infty}) \geq
\frac{\left(3+2\sqrt{2}\right)\vert \Omega \vert}{2M_{1}\left(9+2\sqrt{2}\right)} & \left\Vert a - a_\infty\right\Vert^2_{\mathrm L^1(\Omega)}
+ \frac{\left(3+2\sqrt{2}\right)\vert \Omega \vert}{2M_{1}\left(9+2\sqrt{2}\right)}\left\Vert b - b_\infty\right\Vert^2_{\mathrm L^1(\Omega)}
\\
& + \frac{\left(3+2\sqrt{2}\right)\vert \Omega \vert}{\left(M_{1}+M_{2}\right)\left(9+2\sqrt{2}\right)}\left\Vert c - c_\infty\right\Vert^2_{\mathrm L^1(\Omega)}.
\end{align*}
The proof of the above functional inequality is available in \cite[Lemma 3.3, p.173]{DF06} which exploits the conservation properties \eqref{mass conserve 1} and \eqref{mass conserve 2}.
\end{proof}

\section{The case of $d_c=0$}
In this section, we shall devote our attention to the study of the degenerate model \eqref{eq:model 11} which corresponds to the vanishing of the diffusion coefficient $d_c$. The existence of a smooth positive solution to \eqref{eq:model 11} was proved in \cite[Theorem 3.2]{DF15} for all smooth initial non-negative data if the dimension $N\le3$. Our main objective of this section is to understand the large time behaviour of solutions to \eqref{eq:model 11}. Here, we choose to present our computations when the dimension $N=3$. Analogous results hold true in dimensions one and two as well. Note that we have the entropy equality
\[
\frac{{\rm d}}{{\rm d}t} \left( E(a,b,c) - E(a_{\infty},b_{\infty},c_{\infty}) \right) = - D(a,b,c) \qquad \mbox{ for all }t>0,
\]
where $(a,b,c)$ is the solution to the degenerate model \eqref{eq:model 11} and $(a_{\infty},b_{\infty},c_{\infty})$ is the corresponding equilibrium state given by \eqref{eq:equi-state-1} and \eqref{eq:equi-state-2}. The dissipation functional $D$ in the above equality is given by
\[
D(a,b,c) = 4d_a \int_{\Omega}\left\vert \nabla A \right\vert^2\, {\rm d}x + 4d_b \int_{\Omega}\left\vert \nabla B \right\vert^2\, {\rm d}x + \int_{\Omega} (ab-c)\ln\left(\frac{ab}{c}\right)\, {\rm d}x.
\]
We arrive at the following straightforward lower bound for the dissipation:
\begin{align}\label{eq:dc-0-D-L6}
D(a,b,c) \geq \frac{4d_a}{P(\Omega)} \left\Vert \delta_A \right\Vert^2_{\mathrm L^6(\Omega)} + \frac{4d_b}{P(\Omega)} \left\Vert \delta_B \right\Vert^2_{\mathrm L^6(\Omega)} + 4 \left\Vert AB - C \right\Vert^2_{\mathrm L^2(\Omega)},
\end{align}
thanks to the Poincar\'e-Wirtinger inequality and an algebraic identity which says that for all $p,q\ge0$, there holds $(p-q)\left(\ln p-\ln q\right) \geq 4\left(\sqrt{p}-\sqrt{q}\right)^2$. An application of the H\"older inequality leads to
\begin{align}\label{eq:dc-0-D-L2}
D(a,b,c) \geq \frac{4d_a\left\vert\Omega\right\vert^{-\frac{2}{3}}}{P(\Omega)} \left\Vert \delta_A \right\Vert^2_{\mathrm L^2(\Omega)} + \frac{4d_b\left\vert\Omega\right\vert^{-\frac{2}{3}}}{P(\Omega)} \left\Vert \delta_B \right\Vert^2_{\mathrm L^2(\Omega)} + 4 \left\Vert AB - C \right\Vert^2_{\mathrm L^2(\Omega)}.
\end{align}
It is apparent from the above inequality that the term involving $\left\Vert \delta_C \right\Vert^2_{\mathrm L^2(\Omega)}$ is missing from its lower bound. A similar scenario was handled in the previous section while dealing with the missing $\left\Vert \delta_B \right\Vert^2_{\mathrm L^2(\Omega)}$ term. Our strategy was to derive polynomial (in time) bounds on the supremum norms of the concentrations. Here too, we will adapt a similar approach. Note, however, that the proofs of most results in this section markedly differ from the proofs in the previous section. We begin with a time-dependent $\mathrm L^1(\Omega)$ estimate on certain combinations of the concentrations. This result is inspired by \cite[Theorem 3.1, p.495]{DFMV07}. Our proof argues along similar lines as in \cite{DFMV07} while keeping track of the polynomial (in time) bound.
\begin{Lem}\label{Lem:L1}
Let $N\le 3$ and let $(a,b,c)$ be the solution to the degenerate system \eqref{eq:model 11}. Then, there exist positive constants $\ell_1$ and $\ell_2$ such that for all $t\ge0$, we have
\begin{align*}
\int_0^t \int_\Omega \Big( a^2(s,x) + a(s,x)c(s,x) \Big)\, {\rm d}x\, {\rm d}s & \le \ell_1 \left(1+t\right),
\\
\int_0^t \int_\Omega \Big( b^2(s,x) + b(s,x)c(s,x) \Big)\, {\rm d}x\, {\rm d}s & \le \ell_2 \left(1+t\right).
\end{align*}
\end{Lem}
\begin{proof}
Adding up the equations for $a$ and $c$ in the degenerate model \eqref{eq:model 11}, we obtain
\begin{align}\label{first relation1}
\partial_t \left( a + c \right) - d_a \Delta a = 0.
\end{align}
Let $H\in \mathrm C^\infty_c(\Omega)$ be arbitrary and let $Z:=\frac{a\, d_a}{a+c}$. Now consider the following backward parabolic problem:
\begin{equation}\label{model 31}
\left\{
\begin{aligned}
-\partial_t w - Z \Delta w & = H \, \sqrt{Z} \qquad \mbox{ in }\ \Omega_T,
\\
\nabla w \cdot n(x) & = 0 \qquad \qquad \mbox{ on }\ \partial\Omega_T,
\\
w(T,x) & = 0 \qquad \qquad \ \mbox{ in }\ \Omega.
\end{aligned}
\right.
\end{equation}
In this duality approach, the idea is to multiply the equation \eqref{first relation1} by the solution $w$ to the backward problem \eqref{model 31} followed by an integration over $\Omega_T$ leading to
\begin{align*}
- \int_\Omega w(0,x) \left( a_0(x) + c_0(x) \right)\, {\rm d}x - \int_0^T \int_\Omega \left( a + c \right) \partial_t w \, {\rm d}x\, {\rm d}t - \int_0^T \int_\Omega d_a a \Delta w \, {\rm d}x\, {\rm d}t = 0,
\end{align*}
thanks to integration by parts. Using the equation satisfied by $w$ in \eqref{model 31}, we deduce the following:
\begin{align}\label{first relation2}
\int_0^T \int_\Omega \left( a + c \right)\, H \, \sqrt{Z}\, {\rm d}x\, {\rm d}t 
= \int_\Omega w(0,x) \left( a_0(x) + c_0(x) \right)\, {\rm d}x
\le \left\Vert w(0,\cdot)\right\Vert_{\mathrm L^2(\Omega)} \left\Vert a_0 + c_0 \right\Vert_{\mathrm L^2(\Omega)},
\end{align}
where the inequality is due to the Cauchy-Schwarz inequality. Our next objective is to get an estimate for $\left\Vert w(0,\cdot)\right\Vert_{\mathrm L^2(\Omega)}$. To that end, multiply the evolution equation by $-\Delta w$ and integrating with respect to the spatial variable yields
\[
\int_\Omega \Delta w \partial_t w\, {\rm d}x + \int_\Omega Z \left(\Delta w\right)^2\, {\rm d}x = -\int_\Omega H \sqrt{Z} \Delta w\, {\rm d}x.
\]
Performing an integration by parts in the first term on the left hand side (while using the homogeneous Neumann boundary condition from \eqref{model 31}) and employing Young's inequality for the term on the right hand side yields
\[
-\frac12 \frac{{\rm d}}{{\rm d}t} \int_\Omega \left\vert \nabla w\right\vert^2\, {\rm d}x + \int_\Omega Z \left(\Delta w\right)^2\, {\rm d}x \le \frac12 \int_\Omega H^2\, {\rm d}x + \frac12 \int_\Omega Z \left(\Delta w\right)^2\, {\rm d}x.
\]
Integrating the above inequality on the interval $(0,T)$ in the time variable results in
\begin{align}\label{first relation3}
\int_\Omega \left\vert \nabla w(0,x)\right\vert^2\, {\rm d}x + \int_0^T\int_\Omega Z \left(\Delta w\right)^2\, {\rm d}x\, {\rm d}t \le \int_0^T\int_\Omega H^2\, {\rm d}x\, {\rm d}t.
\end{align}
Hence, invoking Poincar\'e-Wirtinger inequality, we get
\[
\int_\Omega \left( w(0,x) - \frac{1}{\left\vert\Omega\right\vert} \int_\Omega w(0,y)\, {\rm d}y \right)^2\, {\rm d}x \le P(\Omega) \int_\Omega \left\vert \nabla w(0,x)\right\vert^2\, {\rm d}x,
\]
where $P(\Omega)$ is the Poincar\'e constant. Note that integrating the evolution equation in \eqref{model 31} over $\Omega_T$ yields
\begin{align*}
\int_\Omega w(0,x)\, {\rm d}x & = \int_0^T \int_\Omega Z \Delta w \, {\rm d}x\, {\rm d}t + \int_0^T \int_\Omega H \sqrt{Z} \, {\rm d}x\, {\rm d}t
\\
& \le \left( \left( \int_0^T \int_\Omega Z \left(\Delta w\right)^2\, {\rm d}x\, {\rm d}t \right)^\frac12 + \left( \int_0^T \int_\Omega H^2\, {\rm d}x\, {\rm d}t \right)^\frac12 \right) \left\Vert \sqrt{Z} \right\Vert_{\mathrm L^2(\Omega_T)}
\\
& \le 2 \left\Vert H \right\Vert_{\mathrm L^2(\Omega_T)} \left\Vert \sqrt{Z} \right\Vert_{\mathrm L^2(\Omega_T)}
\end{align*}
thanks to the estimate from \eqref{first relation3}. Using the fact that $Z\le d_a$, we deduce
\begin{align}\label{first relation4}
\left(\int_\Omega w(0,x)\, {\rm d}x\right)^2 \le 4 d_a T \left\vert \Omega \right\vert \left\Vert H \right\Vert^2_{\mathrm L^2(\Omega_T)}.
\end{align}
Hence we obtain
\begin{align*}
\int_\Omega \left\vert w(0,x)\right\vert^2\, {\rm d}x 
\le 2 \int_\Omega \left( w(0,x) - \frac{1}{\left\vert\Omega\right\vert} \int_\Omega w(0,y)\, {\rm d}y \right)^2\, {\rm d}x + \frac{2}{\left\vert\Omega\right\vert} \left(\int_\Omega w(0,x)\, {\rm d}x\right)^2
\le \left( 2 P(\Omega) + 8 d_a T \right)\left\Vert H \right\Vert^2_{\mathrm L^2(\Omega_T)},
\end{align*}
thanks to \eqref{first relation3} and \eqref{first relation4}. Going back to \eqref{first relation2}, we have thus obtained
\[
\int_0^T \int_\Omega \left( a + c \right)\, H \, \sqrt{Z}\, {\rm d}x\, {\rm d}t 
\le \left( 2 P(\Omega) + 8 d_a T \right)^\frac12 \left\Vert a_0 + c_0 \right\Vert_{\mathrm L^2(\Omega)} \left\Vert H \right\Vert_{\mathrm L^2(\Omega_T)}.
\]
Since the above inequality holds true for arbitrary $H\in\mathrm C^\infty_c(\Omega)$, we deduce by duality that
\[
\int_0^T \int_\Omega \left( a + c \right)^2 Z\, {\rm d}x\, {\rm d}t \le \left( 2 P(\Omega) + 8 d_a T \right) \left\Vert a_0 + c_0 \right\Vert^2_{\mathrm L^2(\Omega)}.
\]
Substituting for $Z$ in the above inequality, we arrive at
\[
\int_0^T \int_\Omega \left( a^2 + ac \right) \, {\rm d}x\, {\rm d}t \le \left(\frac{2 P(\Omega) + 8 d_a}{d_a}\right) \left\Vert a_0 + c_0 \right\Vert^2_{\mathrm L^2(\Omega)} (1+T).
\]
Proceeding exactly as above but working with the equation satisfied by $b+c$, we can obtain
\[
\int_0^T \int_\Omega \left( b^2 + bc \right) \, {\rm d}x\, {\rm d}t \le \left(\frac{2 P(\Omega) + 8 d_b}{d_b}\right) \left\Vert b_0 + c_0 \right\Vert^2_{\mathrm L^2(\Omega)} (1+T).
\]
This concludes the proof.
\end{proof}
\begin{Lem}\label{L^p estimate d_c}
Let $N\le 3$ and let $(a,b,c)$ be the solution to the degenerate system \eqref{eq:model 11}. Then, there exist positive constants $\ell_3, \ell_4, \ell_7$ such that
\begin{align*}
\left\Vert a(t,\cdot)\right\Vert_{\mathrm L^\frac32(\Omega)} & \le \ell_3 \left(1+t\right)^\frac13 \qquad \mbox{ for } \, t\geq 0,
\\
\left\Vert b(t,\cdot)\right\Vert_{\mathrm L^\frac32(\Omega)} & \le \ell_4 \left(1+t\right)^\frac13 \qquad \mbox{ for } \, t\geq 0,
\\
\left\Vert c(t,\cdot)\right\Vert_{\mathrm L^3(\Omega)} & \le \ell_7 \left(1+t\right) \qquad \, \, \mbox{ for } \, t\geq 0.
\end{align*}
\end{Lem}
\begin{proof}
Multiplying the equation for $a$ in \eqref{eq:model 11} by $a$ and integrating in space and time variables yields
\[
\frac12 \int_\Omega a^2\, {\rm d}x + d_a \int_0^t\int_\Omega \left\vert \nabla a\right\vert^2\, {\rm d}x\, {\rm d}s \le \frac12 \int_\Omega a^2_0\, {\rm d}x + \int_0^t \int_\Omega ac\, {\rm d}x\, {\rm d}s,
\]
where we have used the fact that $a,b$ are non-negative. The estimate from Lemma \ref{Lem:L1} helps us get
\begin{align}\label{eq:Lp-step-1}
\left\Vert a(t,\cdot)\right\Vert^2_{\mathrm L^2(\Omega)} + d_a \left\Vert \nabla a \right\Vert^2_{\mathrm L^2(\Omega_t)} \le \left\Vert a_0\right\Vert^2_{\mathrm L^2(\Omega)} + 2\ell_1 \left(1+t\right).
\end{align}
As 
\[
\frac{\frac13}{1}+\frac{1-\frac13}{2}=\frac{2}{3},
\]
interpolation leads to the following bound:
\[
\left\Vert a(t,\cdot)\right\Vert_{\mathrm L^\frac32(\Omega)} 
\le \left\Vert a(t,\cdot)\right\Vert^\frac13_{\mathrm L^1(\Omega)} \left\Vert a(t,\cdot)\right\Vert^\frac23_{\mathrm L^2(\Omega)}
\le M_1^\frac13 \left\vert\Omega\right\vert^\frac13 \left( \left\Vert a_0\right\Vert^2_{\mathrm L^2(\Omega)} + 2\ell_1 \left(1+t\right) \right)^\frac13
\]
where we have used the mass conservation property \eqref{mass conserve 1}. Taking $\ell_3:= \left(M_1\left\vert\Omega\right\vert\left(\left\Vert a_0\right\Vert^2_{\mathrm L^2(\Omega)} + 2\ell_1\right)\right)^\frac13$, we have thus shown
\[
\left\Vert a(t,\cdot)\right\Vert_{\mathrm L^\frac32(\Omega)} \le \ell_3 \left(1+t\right)^\frac13 \qquad \mbox{ for } \, t\geq 0.
\]
Arguing exactly as above, we obtain
\[
\left\Vert b(t,\cdot)\right\Vert_{\mathrm L^\frac32(\Omega)} \le \ell_4 \left(1+t\right)^\frac13 \qquad \mbox{ for } \, t\geq 0,
\]
with the constant $\ell_4 = \left(M_2\left\vert\Omega\right\vert\left(\left\Vert b_0\right\Vert^2_{\mathrm L^2(\Omega)} + 2\ell_2\right)\right)^\frac13$. It follows from \eqref{eq:Lp-step-1} and the estimate from Lemma \ref{Lem:L1} that
\begin{align}\label{eq:Lp-step-2}
\int_0^t\left\Vert a(s,\cdot)\right\Vert^2_{\mathrm H^1(\Omega)}\, {\rm d}s \le \ell_5 (1+t).
\end{align}
A similar estimate holds for $\left\Vert b\right\Vert^2_{\mathrm L^2(0,t;\mathrm H^1(\Omega))}$ as well. By Sobolev embedding we have
\begin{align}\label{eq:Lp-step-3}
\left\Vert a(t,\cdot)\right\Vert_{\mathrm L^6(\Omega)} \le \ell_6 \left\Vert a(t,\cdot)\right\Vert_{\mathrm H^1(\Omega)},
\qquad
\left\Vert b(t,\cdot)\right\Vert_{\mathrm L^6(\Omega)} \le \ell_6 \left\Vert b(t,\cdot)\right\Vert_{\mathrm H^1(\Omega)}.
\end{align}
Exploiting the non-negativity of $c$, observe from \eqref{eq:model 11} that $c$ satisfies the inequality $\partial_t c \le ab$. Hence we have 
\[
\left( c(t,x)\right)^3
\le 2^2 \left( \left(c_0(x)\right)^3 + \left( \int_0^t a(s,x)\, b(s,x)\, {\rm d}s\right)^3 \right)
\le 2^2(1+t)^2 \left( \left(c_0(x)\right)^3 + \int_0^t \left( a(s,x)\, b(s,x)\right)^3\, {\rm d}s \right),
\]
thanks to Jensen's inequality. Integrating the above inequality in the $x$ variable yields
\[
\left\Vert c(t,\cdot)\right\Vert^3_{\mathrm L^3(\Omega)} 
\le 2^2(1+t)^2 \left( \left\Vert c_0\right\Vert^3_{\mathrm L^3(\Omega)} + \int_0^t \left\Vert ab(s,\cdot) \right\Vert^3_{\mathrm L^3(\Omega)}\, {\rm d}s\right). 
\]
Employing the H\"older inequality leads to the following bound
\begin{align*}
\left\Vert c(t,\cdot)\right\Vert^3_{\mathrm L^3(\Omega)} 
& \le 2^2(1+t)^2 \left( \left\Vert c_0\right\Vert^3_{\mathrm L^3(\Omega)} + \int_0^t \left\Vert a(s,\cdot) \right\Vert^\frac12_{\mathrm L^6(\Omega)}\, \left\Vert b(s,\cdot) \right\Vert^\frac12_{\mathrm L^6(\Omega)} \, {\rm d}s\right) 
\\
& \le 2^2(1+t)^2 \left( \left\Vert c_0\right\Vert^3_{\mathrm L^3(\Omega)} + \frac12 \int_0^t \left( \left\Vert a(s,\cdot) \right\Vert_{\mathrm L^6(\Omega)} + \left\Vert b(s,\cdot) \right\Vert_{\mathrm L^6(\Omega)} \right) \, {\rm d}s\right),
\end{align*}
where we have applied the Young's inequality. Hence it follows from \eqref{eq:Lp-step-2} and \eqref{eq:Lp-step-3} that
\[
\left\Vert c(t,\cdot)\right\Vert_{\mathrm L^3(\Omega)} \le \ell_7 (1+t) \qquad \mbox{ for }t\ge0,
\]
for some constant $\ell_7$
\end{proof}
Our next task is to obtain polynomial (in time) growth estimates on the solution in the supremum norm.
\begin{Prop}\label{L^infty d_c}
Let $N\le 3$ and let $(a,b,c)$ be the solution to the degenerate system \eqref{eq:model 11}. Then, there exist positive constants $K_\infty$ and $\mu$ such that for all $t\ge0$, we have
\begin{align*}
\left\Vert a \right\Vert_{\mathrm L^\infty(\Omega_t)} & \le K_\infty \left(1+t\right)^\mu
\\
\left\Vert b \right\Vert_{\mathrm L^\infty(\Omega_t)} & \le K_\infty \left(1+t\right)^\mu
\\
\left\Vert c \right\Vert_{\mathrm L^\infty(\Omega_t)} & \le K_\infty \left(1+t\right)^\mu
\end{align*}
\end{Prop}
\begin{proof}
Let $G_{d_a}$ denote the Green's function associated with the operator $\partial_t-d_a\Delta$ with Neumann boundary condition. We can express the solution $a$ as follows:
\begin{align}\label{d_a: representation of solution}
a(t,x) = \tilde{a}(t,x) + \int_0^t \int_\Omega G_{d_a}(t-s,x,y)\left[c-ab\right](s,y)\, {\rm d}y\, {\rm d}s,
\end{align}
where $\tilde{a}$ solves the following initial boundary value problem:
\begin{equation*}
\left\{
\begin{aligned}
\partial_t \tilde{a} - d_a \Delta \tilde{a} & = 0 \qquad \ \mbox{ in } \ \Omega_t,
\\
\nabla \tilde{a} \cdot n(x) & = 0 \qquad \ \mbox{ on }\ \partial\Omega_t,
\\
\tilde{a}(0,x) & = a_0 \qquad \mbox{ in } \ \Omega.
\end{aligned}
\right.
\end{equation*}
We recall the following Gaussian bound on the Neumann Green's function (see \cite[Theorem 2.2, p.37]{Morra83}): there exist positive constants $\texttt{C}_{H},\kappa$ such that
\begin{align}\label{heat kernel estimate}
\left\vert G_{d_a}(t-s,x,y) \right\vert \le \texttt{C}_{H} \frac{1}{(t-s)^\frac{N}{2}} e^{-\kappa \frac{\left\vert x-y\right\vert^2}{(t-s)}} =: g(t-s,x-y)
\end{align}
Also see \cite[Theorem 3.1, p.639]{ML15} for general parabolic operators. As a consequence we have the following bound on the solution $\tilde{a}$ to the above homogeneous problem:
\begin{align}\label{eq:tilde-a-bound}
\left\Vert \tilde{a}(t,\cdot)\right\Vert_{\mathrm L^p(\Omega)} \le C_S \left\Vert a_0 \right\Vert_{\mathrm L^p(\Omega)}
\end{align}
for some positive constant $C_{S}$, independent of time, and for any $p\ge1$. In \eqref{d_a: representation of solution}, the positivity of $a$ and $b$ leads to
\[
a(t,x) \le \tilde{a}(t,x) + \int_0^t \int_\Omega G_{d_a}(t-s,x,y) c(s,y)\, {\rm d}y\, {\rm d}s.
\]
Using the aforementioned Gaussian bound, we arrive at
\[
a(t,x) \le \tilde{a}(t,x) + \int_0^t \int_\Omega g(t-s,x-y) c(s,y)\, {\rm d}y\, {\rm d}s.
\]
Computing the $\mathrm L^p$ norm in the $x$ variable, the above inequality leads to
\[
\left\Vert a(t,\cdot) \right\Vert_{\mathrm L^p(\Omega)} \le C_S \left\Vert a_0 \right\Vert_{\mathrm L^p(\Omega)} + \int_0^t \left\Vert g(t-s, \cdot) \right\Vert_{\mathrm L^r(\Omega)} \left\Vert c(s,\cdot) \right\Vert_{\mathrm L^q(\Omega)}\, {\rm d}s,
\]
thanks to the bound \eqref{eq:tilde-a-bound}, the Minkowski's integral inequality and the Young's convolution inequality with
\[
1 + \frac{1}{p} = \frac{1}{r} + \frac{1}{q}.
\]
Therefore, there exists a positive constant $C_{H,N,r}$ such that
\begin{align}\label{Lp-Lq estimate}
\left\Vert a(t,\cdot) \right\Vert_{\mathrm L^p(\Omega)} \le C_S \left\Vert a_0 \right\Vert_{\mathrm L^p(\Omega)} + C_{H,N,r} \int_0^t (t-s)^{-\frac{N}{2}\left(\frac{1}{q} - \frac{1}{p}\right)} \left\Vert c(s,\cdot) \right\Vert_{\mathrm L^q(\Omega)}\, {\rm d}s
\end{align}
Taking $q=3$ and $p=\infty$ in the above bound, we obtain
\begin{align*}
\left\Vert a(t,\cdot) \right\Vert_{\mathrm L^\infty(\Omega)} & \le C_S \left\Vert a_0 \right\Vert_{\mathrm L^p(\Omega)} + C_{H,N,r} \int_0^t (t-s)^{-\frac{N}{6}} \left\Vert c(s,\cdot) \right\Vert_{\mathrm L^3(\Omega)}\, {\rm d}s
\\
& \le C_S \left\Vert a_0 \right\Vert_{\mathrm L^p(\Omega)} + C_{H,N,r} \int_0^t (t-s)^{-\frac{N}{6}} (1+s)\, {\rm d}s,
\end{align*}
thanks to the bound from Lemma \ref{L^p estimate d_c}. Hence we arrive at 
\[
\left\Vert a\right\Vert_{\mathrm L^\infty(\Omega_t)} \le K_\infty (1+t)^\frac{18-N}{6}
\]
for some positive constant $K_\infty$. Arguing along exactly same lines, we can obtain an estimate of $b$ in the supremum norm as well. Again, exploiting the positivity of $c$, we have from \eqref{eq:model 11}
\[
c(t,x) \le c_0(x) + \int_0^t a(s,x) b(s,x)\, {\rm d}s.
\]
The above supremum norm estimates on $a$ and $b$ will help us arrive at the supremum norm estimate for $c$ as well.
\end{proof}
Recall from the lower bound in \eqref{eq:dc-0-D-L2} that the term involving $\left\Vert \delta_C\right\Vert_{\mathrm L^2(\Omega)}$ is apparently missing. Similar to Propositions \ref{main relatio} and \ref{2.0.2}, we now derive a lower bound for the dissipation functional involving this missing term.
\begin{Prop}\label{missing theorem}
Let $N\le 3$ and let $(a,b,c)$ be the solution to the degenerate system \eqref{eq:model 11}. Then the entropy dissipation $D(a,b,c)$ satisfies
\begin{align}\label{eq:main-relatio-dc}
D(a,b,c) \ge \mathcal{K}_c\, (1+t)^{-\frac13} \left( \left\Vert \delta_A \right\Vert^2_{\mathrm L^2(\Omega)} + \left\Vert \delta_B \right\Vert^2_{\mathrm L^2(\Omega)} + \left\Vert \delta_C \right\Vert^2_{\mathrm L^2(\Omega)} \right)
\qquad
\mbox{ for }t\ge0,
\end{align}
where the positive constant $\mathcal{K}_c$ depends only on the domain $\Omega$, the constants $M_1$ and $M_2$ in the mass conservation properties \eqref{mass conserve 1}-\eqref{mass conserve 2} and the nonzero diffusion coefficients $d_a$ and $d_b$.
\end{Prop}
\begin{proof}
It follows from \eqref{eq:dc-0-D-L6} that 
\begin{align}\label{eq:dc-0-D-L6-bis}
D(a,b,c) \geq \frac{4d_a}{P(\Omega)} \left\Vert \delta_A \right\Vert^2_{\mathrm L^6(\Omega)} + \frac{4d_b}{P(\Omega)} \left\Vert \delta_B \right\Vert^2_{\mathrm L^6(\Omega)} + \eta \left\Vert AB - C \right\Vert^2_{\mathrm L^2(\Omega)}
\end{align}
for any $0\le \eta\le 4$. In order to relate the dissipation functional to the missing $\left\Vert \delta_C\right\Vert_{\mathrm L^2(\Omega)}$ term, we work on the following term:
\[
\left\Vert AB - C \right\Vert_{\mathrm L^2(\Omega)}.
\]
Note that we have
\begin{align*}
\left\Vert AB - C \right\Vert_{\mathrm L^2(\Omega)} & = \left\Vert \left( \delta_A + \overline{A} \right) \left( \delta_B + \overline{B} \right) - C \right\Vert_{\mathrm L^2(\Omega)}
\\
& \ge \frac12 \left\Vert \overline{A}\, \overline{B} - C \right\Vert_{\mathrm L^2(\Omega)} - 3 \left\Vert \overline{A}\, \delta_B \right\Vert^2_{\mathrm L^2(\Omega)} - 3 \left\Vert \overline{B}\, \delta_A \right\Vert^2_{\mathrm L^2(\Omega)} - 3 \left\Vert \delta_A\, \delta_B \right\Vert^2_{\mathrm L^2(\Omega)},
\end{align*}
where we have used the following algebraic identities that hold for all $p,q,r\in\mathbb{R}$:
\[
\left( p - q \right)^2 \ge \frac{p^2}{2} - q^2 
\qquad
\mbox{ and }
\qquad 
\left( p + q + r \right)^2 \le 3 \left( p^2 + q^2 + r^2 \right).
\]
Using the mass conservation property \eqref{mass conserve 1}, we have
\[
\int_\Omega A(t,x) \, {\rm d}x \le \left\vert \Omega \right\vert^\frac12 \left( \int_\Omega a(t,x)\, {\rm d}x\right)^\frac12 \le \left\vert \Omega \right\vert^\frac12 M_1^\frac12 \left\vert \Omega \right\vert^\frac12 \implies \overline{A} \le M_1^\frac12.
\]
Similarly, we have $\overline{B}\le M_2^\frac12$. Hence we arrive at
\[
\left\Vert AB - C \right\Vert_{\mathrm L^2(\Omega)} 
\ge \frac12 \left\Vert \overline{A}\, \overline{B} - C \right\Vert_{\mathrm L^2(\Omega)} - 3 M_1^\frac12 \left\Vert \delta_B \right\Vert^2_{\mathrm L^2(\Omega)} - 3 M_2^\frac12 \left\Vert \delta_A \right\Vert^2_{\mathrm L^2(\Omega)} - 3 \left\Vert \delta_A\, \delta_B \right\Vert^2_{\mathrm L^2(\Omega)}.
\]
Employing the H\"older inequality in the last three terms in the above lower bound leads to
\[
\left\Vert AB - C \right\Vert_{\mathrm L^2(\Omega)} 
\ge \frac12 \left\Vert \overline{A}\, \overline{B} - C \right\Vert_{\mathrm L^2(\Omega)} 
- 3 M_1^\frac12 \left\vert \Omega \right\vert^\frac23 \left\Vert \delta_B \right\Vert^2_{\mathrm L^6(\Omega)}
- 3 M_2^\frac12 \left\vert \Omega \right\vert^\frac23 \left\Vert \delta_A \right\Vert^2_{\mathrm L^6(\Omega)}
- 3 \left\Vert \delta_A^2 \right\Vert_{\mathrm L^\frac32(\Omega)} \left\Vert \delta_B \right\Vert^2_{\mathrm L^6(\Omega)}
\]
Using the algebraic identity $(p-q)^2\le 2(p^2+q^2)$ in the last term of the above lower bound, we obtain
\begin{align*}
\left\Vert AB - C \right\Vert_{\mathrm L^2(\Omega)} 
\ge \frac12 \left\Vert \overline{A}\, \overline{B} - C \right\Vert_{\mathrm L^2(\Omega)} 
& - 3 M_1^\frac12 \left\vert \Omega \right\vert^\frac23 \left\Vert \delta_B \right\Vert^2_{\mathrm L^6(\Omega)}
- 3 M_2^\frac12 \left\vert \Omega \right\vert^\frac23 \left\Vert \delta_A \right\Vert^2_{\mathrm L^6(\Omega)}
\\
& - 6 \left\Vert A^2 + \overline{A}^2 \right\Vert_{\mathrm L^\frac32(\Omega)} \left\Vert \delta_B \right\Vert^2_{\mathrm L^6(\Omega)}.
\end{align*}
By employing triangular inequality in the last term of the above lower bound, we arrive at
\begin{align*}
\left\Vert AB - C \right\Vert_{\mathrm L^2(\Omega)} 
& \ge \frac12 \left\Vert \overline{A}\, \overline{B} - C \right\Vert_{\mathrm L^2(\Omega)} 
- \left\vert\Omega \right\vert^\frac23 \left( 3 M_1^\frac12 + 6M_1\right) \left\Vert \delta_B \right\Vert^2_{\mathrm L^6(\Omega)}
- 3 M_2^\frac12 \left\vert \Omega \right\vert^\frac23 \left\Vert \delta_A \right\Vert^2_{\mathrm L^6(\Omega)}
\\
& \qquad \qquad \qquad \qquad - 6 \left\Vert a \right\Vert_{\mathrm L^\frac32(\Omega)} \left\Vert \delta_B \right\Vert^2_{\mathrm L^6(\Omega)}
\\
& \ge \frac12 \left\Vert \overline{A}\, \overline{B} - C \right\Vert_{\mathrm L^2(\Omega)} 
- \left( \left\vert\Omega \right\vert^\frac23 \left( 3 M_1^\frac12 + 6M_1\right) + 6 \ell_3 \left(1+t\right)^\frac13 \right) \left\Vert \delta_B \right\Vert^2_{\mathrm L^6(\Omega)}
- 3 M_2^\frac12 \left\vert \Omega \right\vert^\frac23 \left\Vert \delta_A \right\Vert^2_{\mathrm L^6(\Omega)},
\end{align*}
thanks to the estimate on $\left\Vert a\right\Vert_{\mathrm L^\frac32(\Omega)}$ from Lemma \ref{L^p estimate d_c}. Now we claim that
\[
\left\Vert \overline{A}\, \overline{B} - C \right\Vert_{\mathrm L^2(\Omega)} \ge \left\Vert C - \overline{C} \right\Vert_{\mathrm L^2(\Omega)}.
\]
To see this, factorising $C$ as $\overline{A}\, \overline{B}(1+\mu(x))$, we get
\[
\left\Vert \overline{A}\, \overline{B} - C \right\Vert^2_{\mathrm L^2(\Omega)} =  \overline{A}^2 \,  \overline{B}^2\, \overline{\mu^2}\, \left\vert \Omega\right\vert
\qquad
\mbox{ and }
\qquad
\left\Vert C - \overline{C} \right\Vert^2_{\mathrm L^2(\Omega)} = \overline{A}^2 \,  \overline{B}^2 \left\Vert \mu - \overline{\mu}\right\Vert^2_{\mathrm L^2(\Omega)} \le \overline{A}^2 \,  \overline{B}^2\, \overline{\mu^2}\, \left\vert \Omega\right\vert.
\]	
Putting it all together, we arrive at
\[
\left\Vert AB - C \right\Vert_{\mathrm L^2(\Omega)} 
\ge \frac12 \left\Vert \delta_C \right\Vert^2_{\mathrm L^2(\Omega)}
- \left( \left\vert\Omega \right\vert^\frac23 \left( 3 M_1^\frac12 + 6M_1\right) + 6 \ell_3 \left(1+t\right)^\frac13 \right) \left\Vert \delta_B \right\Vert^2_{\mathrm L^6(\Omega)}
- 3 M_2^\frac12 \left\vert \Omega \right\vert^\frac23 \left\Vert \delta_A \right\Vert^2_{\mathrm L^6(\Omega)}.
\]
Hence it follows from \eqref{eq:dc-0-D-L6-bis} that
\begin{equation}\label{eq:dc-0-D-L6-ter}
\begin{aligned}
D(a,b,c) \geq \left( \frac{4d_a}{P(\Omega)} - 3 \eta\, M_2^\frac12 \left\vert \Omega \right\vert^\frac23 \right)& \left\Vert \delta_A \right\Vert^2_{\mathrm L^6(\Omega)} 
+ \frac{\eta}{2} \left\Vert \delta_C \right\Vert^2_{\mathrm L^2(\Omega)}
\\
& + \left( \frac{4d_b}{P(\Omega)} - \eta\, \left( \left\vert\Omega \right\vert^\frac23 \left( 3 M_1^\frac12 + 6M_1\right) + 6 \ell_3 \left(1+t\right)^\frac13 \right) \right) \left\Vert \delta_B \right\Vert^2_{\mathrm L^6(\Omega)}.
\end{aligned}
\end{equation}
Let us take
\[
\eta(t) := \frac{4\min\left\{d_a, d_b, 1\right\}}{P(\Omega) \left(3 M_2^\frac12 \left\vert \Omega \right\vert^\frac23 + \left\vert\Omega \right\vert^\frac23 \left( 3 M_1^\frac12 + 6M_1\right) + 6 \ell_3 \right) + 1} (1+t)^{-\frac13}.
\]
Then, it follows from \eqref{eq:dc-0-D-L6-ter} that the dissipation functional has the following lower bound:
\begin{align}\label{eq:dc-0-D-L6-quad}
D(a,b,c) \ge \frac12 \eta(t) \left\Vert \delta_C \right\Vert^2_{\mathrm L^2(\Omega)}.
\end{align}
Observe that the above choice of $\eta$ clearly satisfies $0\le \eta \le 4$. To conclude our proof, we write
\begin{align*}
D(a,b,c) & = \frac12 D(a,b,c) + \frac12 D(a,b,c)
\\
& \ge \frac{2d_a\left\vert\Omega\right\vert^{-\frac{2}{3}}}{P(\Omega)} \left\Vert \delta_A \right\Vert^2_{\mathrm L^2(\Omega)} 
+ \frac{2d_b\left\vert\Omega\right\vert^{-\frac{2}{3}}}{P(\Omega)} \left\Vert \delta_B \right\Vert^2_{\mathrm L^2(\Omega)}
+ \frac14 \eta(t) \left\Vert \delta_C \right\Vert^2_{\mathrm L^2(\Omega)},
\end{align*}
where we have used the lower bound from \eqref{eq:dc-0-D-L6-bis} for the first term while the lower bound from \eqref{eq:dc-0-D-L6-quad} for the second term. Hence we have proved \eqref{eq:main-relatio-dc} with the constant $\mathcal{K}_c$ given by
\[
\mathcal{K}_c := \min \left\{ \frac{2d_a\left\vert\Omega\right\vert^{-\frac{2}{3}}}{P(\Omega)}, \frac{2d_b\left\vert\Omega\right\vert^{-\frac{2}{3}}}{P(\Omega)}, \frac{\min\left\{d_a, d_b, 1\right\}}{P(\Omega) \left(3 M_2^\frac12 \left\vert \Omega \right\vert^\frac23 + \left\vert\Omega \right\vert^\frac23 \left( 3 M_1^\frac12 + 6M_1\right) + 6 \ell_3 \right) + 1}\right\}.
\]
\end{proof}
Thanks to the lower bound in \eqref{eq:main-relatio-dc}, we can derive a sub-exponential decay estimate for the relative entropy.
\begin{Prop}\label{decay 1 d_c}
Let $N\leq3$ and let $(a,b,c)$ be the solution to the degenerate system \eqref{eq:model 11}. Let $(a_\infty,b_\infty,c_\infty)$ be the associated equilibrium state given by \eqref{eq:equi-state-1}-\eqref{eq:equi-state-2}. Then, for any given positive $\varepsilon \ll1$, there exists a time $T_{\varepsilon}$ and two positive constants $\mathcal{S}_5$ and $\mathcal{S}_6$ such that
\begin{align}\label{eq:sub-exp-decay-relative-entropy-dc}
E(a,b,c)- E(a_{\infty},b_{\infty},c_{\infty})
\leq \mathcal{S}_5 \, e^{-\mathcal{S}_6(1+t)^{\frac{2-\epsilon}{3}}} \qquad \mbox{ for }\, t\ge T_\varepsilon.
\end{align}
\end{Prop} 
The proof of the above proposition is exactly similar to the proof of Proposition \ref{decay 1}. Hence we skip the proof. Note that the constants $\mathcal{S}_5$ and $\mathcal{S}_6$ appearing in \eqref{eq:sub-exp-decay-relative-entropy-dc} depend on the constant $\mathcal{K}_c$ (appearing in Proposition \ref{eq:main-relatio-dc}), the constants $K_\infty$ and $\mu$ (appearing in Proposition \ref{L^infty d_c}). Finally, we are all equipped to prove our main result of this section.
\begin{proof}[Proof of Theorem \ref{convergence theorem 2}]
We have already obtained sub-exponential decay (in time) of the relative entropy \eqref{eq:sub-exp-decay-relative-entropy-dc} in Proposition \ref{decay 1 d_c}. Hence the sub-exponential decay in the $\mathrm L^1$-norm is a direct consequence of the following Czisz\'ar-Kullback-Pinsker type inequality that relates relative entropy and the $\mathrm L^1$-norm:
\begin{align*}
E(a,b,c) - E(a_{\infty},b_{\infty},c_{\infty}) \geq
\frac{\left(3+2\sqrt{2}\right)\vert \Omega \vert}{2M_{1}\left(9+2\sqrt{2}\right)} & \left\Vert a - a_\infty\right\Vert^2_{\mathrm L^1(\Omega)}
+ \frac{\left(3+2\sqrt{2}\right)\vert \Omega \vert}{2M_{1}\left(9+2\sqrt{2}\right)}\left\Vert b - b_\infty\right\Vert^2_{\mathrm L^1(\Omega)}
\\
& + \frac{\left(3+2\sqrt{2}\right)\vert \Omega \vert}{\left(M_{1}+M_{2}\right)\left(9+2\sqrt{2}\right)}\left\Vert c - c_\infty\right\Vert^2_{\mathrm L^1(\Omega)}.
\end{align*}
The proof of the above functional inequality is available in \cite[Lemma 3.3, p.173]{DF06} which exploits the conservation properties \eqref{mass conserve 1} and \eqref{mass conserve 2}.
\end{proof}

\appendix

\section{some useful results}\label{sec:app}
\begin{Lem}[Poincar\'e-Wirtinger inequality]\label{Poincare-Wirtinger}
There exists a positive constant $P(\Omega)$, depending only on $\Omega$ and $q$, such that
\[
P(\Omega) \left\Vert \nabla f\right\Vert^2_{\mathrm L^2(\Omega)} \ge \left\Vert f - \overline{f} \right\Vert^2_{\mathrm L^q(\Omega)} \qquad \mbox{ for all }f\in\mathrm H^1(\Omega),
\]
where
\begin{equation*}
q = 
\left\{
\begin{aligned}
\frac{2N}{N-2} & \qquad \mbox{ for }\, \, N \ge 3,
\\
\in [1,\infty) & \qquad \mbox{ for }\, \, N=2,
\\
\in [1,\infty)\cup\{\infty\} & \qquad \mbox{ for }\, \, N=1.
\end{aligned}
\right.
\end{equation*}
We refer to $P(\Omega)$ as the Poincar\'e constant.
\end{Lem}
\begin{Thm}[Second Order Regularity and Integrability estimation] \label{estimation 1}
Let $d>0$ and let $\tau\in[0,T)$. Take $\theta \in \mathrm L^p(\Omega_{\tau,T})$ for some $1<p< +\infty$. Let $\psi$ be the solution to the backward heat equation:
\begin{equation*}
\left\{
\begin{aligned}
\partial_t \psi + d \Delta \psi & = -\theta \qquad \mbox{ for }(t,x)\in\Omega_{\tau,T},
\\
\nabla \psi \cdot n(x) & = 0 \qquad \ \ \mbox{ for }(t,x)\in[\tau,T]\times\partial\Omega,
\\
\psi(T,x) & = 0 \qquad \ \  \mbox{ for }x\in\Omega.
\end{aligned}
\right.
\end{equation*}
Then, there exists a positive constant $C_{SOR}$, depending only on the domain $\Omega$, the dimension $N$ and the exponent $p$ such that the following maximal regularity holds:
\begin{align}\label{eq:estimate-laplacian}
\Vert \Delta \psi \Vert_{\mathrm L^p(\Omega _{\tau,T})} \leq \frac{C_{SOR}}{d} \Vert \theta \Vert_{\mathrm L^p(\Omega_{\tau,T})}.
\end{align}
Moreover, if \ $\theta\geq 0$ then $\psi(t,x)\geq 0$ for almost every $(t,x)\in\Omega_{\tau,T}$. Furthermore, we have
\begin{equation}\label{eq:Ls-estimates}
\begin{aligned}
\mbox{If }\, & p<\frac{N+2}{2}\, \mbox{ then } \qquad \left\Vert \psi \right\Vert_{\mathrm L^s(\Omega_{\tau,T})} \le C_{IE} \left\Vert \theta \right\Vert_{\mathrm L^p(\Omega_{\tau,T})} \qquad \mbox{ for all }\, \, s < \frac{(N+2)p}{N+2-2p}
\\
\mbox{If }\, & p=\frac{N+2}{2}\, \mbox{ then } \qquad \left\Vert \psi \right\Vert_{\mathrm L^s(\Omega_{\tau,T})} \le C_{IE} \left\Vert \theta \right\Vert_{\mathrm L^p(\Omega_{\tau,T})} \qquad \mbox{ for all }\, \, s < \infty
\end{aligned}
\end{equation}
where the constant $C_{IE}=C_{IE}(T-\tau, \Omega,d,p,s)$ and
\begin{align}\label{eq:Linfty-estimate}
\mbox{if }\, & p>\frac{N+2}{2}\, \mbox{ then } \qquad\left\Vert \psi \right\Vert_{\mathrm L^\infty(\Omega _{\tau,T})} \leq  C_{IE} \left\Vert \theta \right\Vert_{\mathrm L^p(\Omega_{\tau,T})}.
\end{align}
\end{Thm}
The proof of \eqref{eq:estimate-laplacian} can be found in \cite[Theorem 1]{Lam87}. We refer to the constant $C_{SOR}$ as the second order regularity constant. Proof of the estimates \eqref{eq:Ls-estimates} can be found in \cite[Lemma 3.3]{CDF14} and the estimate \eqref{eq:Linfty-estimate} was derived in \cite[Lemma 4.6]{Tan18}. We refer to the constant $C_{IE}$ as the integrability estimation constant.
\begin{Thm}[$p^{\textrm{th}}$ order integrability estimation]\label{PE}
Let $p\in(2,\infty)$ and let $p'$ be its H\"older conjugate. Let $M(t,x)$ be such that the following holds
\[
\theta\leq M(t,x) \leq \Theta \quad \forall (t,x)\in\Omega_T,
\]
for some fixed positive constants $\theta,\Theta$. Let $\psi_0\in\mathrm L^p(\Omega)$ and let $\psi$ be a weak solution to
\begin{equation*}
\left\{
\begin{aligned}
\partial_t \psi - \Delta \left( M \psi\right) & = 0 \qquad \qquad \mbox{ for }(t,x)\in\Omega_T,
\\
\nabla \psi \cdot n(x) & = 0 \qquad \qquad \mbox{ for }(t,x)\in[0,T]\times\partial\Omega,
\\
\psi(0,x) & = \psi_0(x) \qquad \mbox{ for }x\in \Omega.
\end{aligned}
\right.
\end{equation*}
Then the following estimate holds
\[
\left\Vert \psi \right\Vert_{\mathrm L^p(\Omega_T)} \le \left( 1 + \Theta K_{\theta,\Theta,p'}\right) T^\frac{1}{p} \left\Vert \psi_0 \right\Vert_{\mathrm L^p(\Omega)},
\]
where the constant $K_{\theta,\Theta,p'}$ is given by
\[
K_{\theta,\Theta,p'} := \frac{C^{PRC}_{\frac{\theta+\Theta}{2},p'}\left(\frac{\Theta-\theta}{2}\right)}{1 - C^{PRC}_{\frac{\theta+\Theta}{2},p'}\left(\frac{\Theta-\theta}{2}\right)}
\qquad
\mbox{ provided we have }
\qquad
C^{PRC}_{\frac{\theta+\Theta}{2},p'}\left(\frac{\Theta-\theta}{2}\right) < 1.
\]
Here, the constant $C^{PRC}_{r,p'}$ is the best constant in the following parabolic regularity estimate:
\[
\left\Vert \Delta \phi \right\Vert_{\mathrm L^{p'}(\Omega_T)} \le C^{PRC}_{r,p'} \left\Vert f \right\Vert_{\mathrm L^{p'}(\Omega_T)},
\]
where $\phi,f:[0,T]\times\Omega\to\mathbb{R}$ are any two functions such that $f\in \mathrm L^{p'}(\Omega_T)$ and they satisfy
\begin{equation*}
\left\{
\begin{aligned}
\partial_t \phi + r \Delta \phi & = f \qquad \mbox{ for }(t,x)\in\Omega_T,
\\
\nabla \phi \cdot n(x) & = 0 \qquad \mbox{ for }(t,x)\in[0,T]\times\partial\Omega,
\\
\phi(T,x) & = 0 \qquad \mbox{ for }x\in\Omega.
\end{aligned}
\right.
\end{equation*}
\end{Thm}
It has to be noted that $C^{PRC}_{r,p^{'}}< \infty$ for $r>0$ and $C^{PRC}_{r,2} \leq \frac{1}{r}$\ and depends  only on $r,p^{'}$, the domain and on the dimension, i.e it is independent of time. Moreover, as $C^{PRC}_{r,p^{'}}< \infty$, if we take the difference between $\theta$ and $\Theta$ sufficiently small, then we have the required property that 
\[
C^{PRC}_{\frac{\theta+\Theta}{2},p'}\left(\frac{\Theta-\theta}{2}\right) < 1.
\]
Proof of the above theorem can be found in \cite[Proposition 1.1]{CDF14}.

\bibliography{ref3.bib}

\end{document}